\documentclass[12pt]{article}

\usepackage[margin=1in,
includefoot,
headsep=20pt 
]{geometry}
\usepackage{amsmath,amssymb,amsthm}
\usepackage{graphicx,verbatim}
\usepackage{algorithm,algorithmic}
\usepackage[small,bf]{caption}
\usepackage{subcaption} 
\usepackage{cases}
\usepackage[table,dvipsnames]{xcolor}
\usepackage[shortlabels]{enumitem}
\usepackage{hyperref,url}
\usepackage{cite}
\hypersetup{
	colorlinks=true,
	linkcolor=red,     
	urlcolor=magenta,
	citecolor={blue},
}
\usepackage{booktabs,adjustbox,multirow}  
\usepackage{import} 
\usepackage{aligned-overset} 
\usepackage[toc,page]{appendix}
\usepackage[normalem]{ulem}
\usepackage{tikz,psfrag,pgfplots}
\pgfplotsset{compat=newest}
\usetikzlibrary{plotmarks,arrows.meta}
\usepgfplotslibrary{patchplots,external}
\usepackage{lmodern}
\usepackage{dsfont} 
\usepackage{bm,bbm} 
\usepackage{upgreek} 

\newcommand{\grad}{{\nabla}}   
\newcommand{\zero}{\mathbf{0}}  
\newcommand{\one}{\mathbf{1}}   
\newcommand{\real}{\mathbb{R}}  

\newcommand{\col}{\mathrm{col}}     

\newcommand{\Ex}{\mathop{\mathds{E}{}}}

\newcommand{\minimize}{\mathop{\text{minimize}}}

\newcommand{\eg}{{\it e.g.}}
\newcommand{\ie}{{\it i.e.}}

\renewcommand{\top}{\textit{\footnotesize \texttt{T}}} 
\newcommand{\define}{\triangleq} 
\newcommand{\norm}[1]{\lVert #1\rVert}

\def\bxi{\boldsymbol{\xi}}

\def\F{{\mathbf{F}}}

\def\I{{\mathbf{I}}}

\def\P{{\mathbf{P}}}

\def\R{{\mathbf{R}}}

\def\W{{\mathbf{W}}}

\def\f{{\mathbf{f}}}
\def\g{{\mathbf{g}}}

\def\q{{\mathbf{q}}}

\def\u{{\mathbf{u}}}
\def\v{{\mathbf{v}}}

\def\x{{\mathbf{x}}}
\def\y{{\mathbf{y}}}

\newcommand{\cF}{{\mathcal{F}}}

\def\evdots{\vbox{\baselineskip=2pt \lineskiplimit=0pt 
		\kern6pt \hbox{$.$}\hbox{$.$}\hbox{$.$}}}  
\newtheorem{lemma}{{Lemma}}
\newtheorem{assumption}{{Assumption}}
\newtheorem{theorem}{{Theorem}}
\newtheorem{corollary}{{Corollary}}
\newtheorem{proposition}{{Proposition}}

\newtheorem{remark}{{Remark}}

\usepackage{sectsty}
\allsectionsfont{\sffamily}
\usepackage{fancyhdr}

\begin{document}

\title{\bfseries   Stochastic Gradient Tracking over Time-Varying Networks: One-Step Lyapunov Analysis}
\author{\sffamily Sulaiman A. Alghunaim \\
{\small Department of Electrical Engineering, Kuwait University}
\\ {\tt\small sulaiman.alghunaim@ku.edu.kw}
}

\date{}
 \maketitle
\thispagestyle{fancy} 


	\begin{abstract}
		We study decentralized stochastic gradient tracking over a time-varying
		network of $N$ agents under a uniform window-mixing condition. Products of $\tau$
		consecutive doubly stochastic mixing matrices contract disagreement by a
		factor $\lambda<1$, although individual matrices need not contract
		disagreement strictly and individual communication graphs may be
		disconnected. We construct a time-varying quadratic norm that turns this
		window contraction into an exact one-step Lyapunov identity. This leads to
		coupled one-step recursions for the centroid and disagreement errors,
		without unrolling the dynamics over communication windows. For smooth
		strongly convex objectives, the leading stochastic term is
		$\widetilde{\mathcal O}(1/(NK))$; for smooth convex objectives, it is
		$\mathcal O(1/\sqrt{NK})$. Both match their centralized mini-batch
		counterparts and yield linear speedup after a network-dependent transient.
	\end{abstract}
\section{Introduction}
\label{sec:intros}
In decentralized optimization, a group of agents minimizes a global
objective by combining local computation with communication restricted to
network neighbors; no central node collects the data or the iterates
\cite{nedic2009distributed,sayed2014nowbook}. Problems of this form arise in
multi-agent coordination, formation control, and flocking
\cite{jadbabaie2003coordination,olfatisaber2006flocking,shorinwa2024distributed},
in distributed model-predictive control \cite{mota2012distributedadmm}, and
in distributed estimation and adaptive processing over sensor networks
\cite{schizas2007consensus,sayed2014nowbook}. The time-varying regime is the
practically important one: mobility, scheduling, link failures, and
bandwidth constraints all cause the communication graph to change from one
iteration to the next.

\subsection{Problem and algorithm} 
\noindent{\bfseries\small Problem.} We study the following distributed stochastic optimization problem: \begin{equation} 
	\label{min_main}
	 \minimize_{x\in\real^d} \quad f(x) \define \frac{1}{N}\sum_{i=1}^{N} f_i(x), \qquad f_i(x) \define \Ex\left[F_i(x;\xi_i)\right], 
	\end{equation} over a decentralized network of $N$ agents. Each agent $i$ has access only to its private local objective $f_i:\real^d\to\real$ and to stochastic samples $\xi_i$ used to evaluate gradients of the loss function $F_i(\cdot;\xi_i)$; the stochasticity may represent local data sampling or measurement noise. By exchanging information only with their current neighbors, the agents seek a common solution to \eqref{min_main} without requiring a central coordinator.

\noindent{\bfseries\small TV-GT algorithm.} Gradient tracking \cite{xu2015augmented,di2016next,nedic2017achieving} augments decentralized gradient descent \cite{nedic2009distributed,sayed2014nowbook} with a dynamic estimate of the aggregate gradient to eliminate a bias caused by differences among local objectives. In this work, we study stochastic gradient tracking (GT) over  time-varying communication graphs. At iteration $k$, agent $i$ communicates with its current neighbors according to the mixing weights $\{w_{ij}^k\}_{j=1}^{N}$, where $w_{ij}^k>0$ only if agent $j$ is a neighbor of agent $i$ at time $k$ or $j=i$, and $w_{ij}^k=0$ otherwise. Each agent $i$ maintains a local primal variable $x_i^k\in\real^d$ and a gradient-tracking variable $g_i^k\in\real^d$, and performs \begin{subequations}
	 \label{tv_gt_algorithm_agent} 
	 \begin{align} x_i^{k+1} &= \sum_{j=1}^{N} w_{ij}^k ( x_j^k-\alpha g_j^k ) \label{tv_gt_algorithm_agent_x} \\ g_i^{k+1} &= \sum_{j=1}^{N} w_{ij}^k g_j^k + \grad F_i(x_i^{k+1};\xi_i^{k+1}) - \grad F_i(x_i^k;\xi_i^k), \label{tv_gt_algorithm_agent_g}
	  \end{align} 
  \end{subequations}
   with initialization $g_i^0 = \grad F_i(x_i^0;\xi_i^0)$ and $x_i^0\in\real^d$.  The communication graph and the mixing weights $\{w_{ij}^k\}$ may vary with $k$ and are assumed to satisfy Assumption~\ref{assumption:tv_gt_network}.

\subsection{Related work}
\label{sec:related_work}
Decentralized optimization over time-varying networks dates back to early
distributed subgradient methods. The works
\cite{nedic2009distributed,ram2010distributed} analyze distributed
(sub)gradient methods over graph sequences for which connectivity is required
only over bounded communication windows, while the individual graphs may be
disconnected. Most later work on time-varying networks adopts this window-based
connectivity model. More recently,
\cite{koloskova2020unified} analyzes decentralized SGD (DSGD) under general
time-varying topologies and shows that DSGD achieves a leading variance term reduced by a factor proportional
to the number of agents, \ie, linear speedup. DSGD, however, does not
correct the bias induced by heterogeneous local objectives. Consequently,
with constant stepsizes its error generally contains a persistent
heterogeneity-dependent term, while exact convergence typically requires
diminishing stepsizes \cite{yuan2020influence}.

Gradient tracking (GT) removes this steady-state bias by
dynamically tracking the aggregate gradient \cite{xu2015augmented,di2016next}. Many works have analyzed GT over time-varying networks
\cite{di2016next,nedic2017achieving,sundararajan2020analysis,maros2020geometrically}. These works, however, focus on deterministic gradients and establish exact convergence under
window-based connectivity, but do not address our stochastic
gradient noise setting.

Recent work has begun to address stochastic GT over time-varying graphs
\cite{huang2025beyond,nguyen2025graphs,fainman2026convergence}.
The work~\cite{huang2025beyond} imposes an expected-connectivity condition
at each communication round, whereas \cite{nguyen2025graphs}
considers graph sequences with finite-time consensus. The work
\cite{fainman2026convergence} replaces exact finite-time consensus by an
approximate variant under symmetric doubly stochastic mixing. These settings
do not cover Assumption~\ref{assumption:tv_gt_network}, which allows
nonsymmetric mixing matrices, disconnected individual graphs, and an
arbitrary window contraction factor $\lambda<1$.

Directed networks, where double stochasticity is generally unavailable, are
handled by a modified family of GT methods
\cite{scutari2019distributed,saadatniaki2020decentralized}.
The work \cite{scutari2019distributed} extends NEXT to time-varying directed
graphs through the push-sum-based SONATA framework, while
\cite{saadatniaki2020decentralized} extends push--pull methods
\cite{xin2018linear,pu2020push} to time-varying directed networks and
establishes linear convergence for strongly convex objectives; see also
\cite{nedic2025ab}. These results, however, focus on deterministic gradients.

Stochastic GT over time-varying directed graphs has also been considered
recently \cite{chen2024accelerated,nguyen2024decentralized}. Push-sum,
gradient tracking, and momentum are combined in \cite{chen2024accelerated},
under the requirement that the communication graph be strongly connected at
every iteration; \cite{nguyen2024decentralized} pairs push--pull gradient
tracking with heavy-ball momentum and likewise assumes strong connectivity
of every individual graph. Neither analysis covers the mixing model of
Assumption~\ref{assumption:tv_gt_network}, which permits individual graphs
to be disconnected.

Beyond the differences in network assumptions and convergence guarantees,
our analysis also differs from prior work in how the time-varying network is
handled. Analyses based directly on window contraction typically unroll the
error over complete communication windows while also accounting for the
associated partial-window terms; see, \eg,
\cite{koloskova2020unified,fainman2026convergence}. We instead work in a time-varying quadratic norm under which the
window-mixing condition of Assumption~\ref{assumption:tv_gt_network} already
forces a strict decrease at every single iteration. This construction builds on
classical converse-Lyapunov theory for stable time-varying systems and is
closely related to the Lyapunov--Stein equation and converse Lyapunov results
for linear time-varying systems
\cite{stein1952some,anderson1981detectability,geiselhart2014alternative}.
Related ideas have also been used in networked control to combine joint
connectivity with time-varying quadratic Lyapunov functions
\cite{su2012stability,liu2017adaptive,liu2025distributed}. However, these works
consider {\em continuous-time} dynamics and problems different from
\eqref{min_main}. Here, we apply this construction directly to the
discrete-time disagreement mixing process underlying stochastic gradient
tracking and use the resulting time-varying quadratic norm to control the
coupled primal--tracking errors.

\subsection{Contributions}
\label{sec:contributions}

This paper analyzes stochastic gradient tracking
\eqref{tv_gt_algorithm_agent} under the uniform window-mixing condition in
Assumption~\ref{assumption:tv_gt_network}: every product of $\tau$
consecutive mixing matrices contracts disagreement by a factor
$\lambda<1$, while the individual matrices are required only to be
nonnegative and doubly stochastic. In particular, individual communication
graphs may be disconnected and need not provide contraction at every
iteration.

\begin{itemize}
	
	\item
	\textbf{One-step analysis under window contraction.}
	We associate the time-varying mixing process with a quadratic Lyapunov norm
	whose defining identity gives a strict decrease at every iteration. This
	holds even when the individual mixing matrices are not strictly contractive.
	After scaling the tracking variable appropriately, the same norm controls
	the primal and tracking disagreements jointly. The resulting analysis works
	directly at the iteration level and avoids unrolling the recursion over
	communication windows. This route differs from existing window-based analyses; we are not aware
	of prior work that applies this converse-Lyapunov construction to
	stochastic gradient tracking over time-varying networks.
	
\item
\textbf{Finite-time guarantees for strongly convex and convex objectives.}
For smooth strongly convex objectives, the leading stochastic term is
$\widetilde{\mathcal O}(\sigma^2/(NK))$, where $\sigma^2$ bounds the
gradient-noise variance. For smooth convex objectives, we obtain an ergodic
bound with leading term $\mathcal O(\sigma/\sqrt{NK})$. These terms match the
corresponding centralized mini-batch rates and yield linear speedup in the
number of agents.
	
\item
\textbf{Explicit network dependence.}
We characterize explicitly how the convergence bounds depend on the
window length $\tau$ and the window contraction factor $\lambda$. In the
strongly convex case, the coefficient of the higher-order stochastic
term scales as $
\frac{\tau^3}{(1-\lambda)^3}$,
and the admissible stepsize and transient length are also characterized
explicitly in terms of $\tau$ and $1-\lambda$. These expressions make explicit how long the network-dependent transient
lasts before the centralized rate takes over.
	
\end{itemize}

\section{Assumptions and error dynamics}
In this section, we first state the assumptions on the objective functions, communication
network, and stochastic gradients. We then derive the centroid and
disagreement dynamics that form the basis of the convergence analysis.

\subsection{Assumptions}

\begin{assumption}[\bf\small Functions] \rm
	\label{assumption:tv_gt_functions}
	Each local objective $f_i$ is $L$-smooth and $\mu$-strongly convex
	for some $0\leq\mu\leq L$, where $\mu=0$ corresponds to the convex case.
	We assume that $f$ admits a minimizer and let $x^\star$ denote any
	minimizer of $f$. For the strongly convex case ($\mu>0$), we define
	the condition number $
	\kappa\define\frac{L}{\mu}$.
\end{assumption}

\begin{assumption}[\bf\small Network] \rm
	\label{assumption:tv_gt_network}
	The sequence $\{W_k\}_{k\geq0}$ is deterministic, and each
	$W_k\in\real^{N\times N}$ is nonnegative and doubly stochastic ($w_{ij}^k\geq 0$, $
		W_k\one=\one$,
		$\one^\top W_k=\one^\top$).
	Define
	\begin{equation}
		\Pi
		\define
		\frac{1}{N}\one\one^\top,
		\qquad
		\widetilde W_k
		\define
		W_k-\Pi.
	\end{equation}
	There exist an integer $\tau\geq1$ and a constant
	$\lambda\in[0,1)$ such that every product of $\tau$ consecutive
	mixing matrices contracts disagreement:
	\begin{equation}
		\label{tv_gt_window_contraction}
		\left\|
		\widetilde W_{k+\tau-1:k}
		\right\|_2
		\leq
		\lambda,
		\qquad
		\text{for all }k\geq0,
	\end{equation}
	where $\widetilde W_{t:s}\define\widetilde W_t
		\widetilde W_{t-1}
		\cdots
		\widetilde W_s$, $t\geq s$. We define the
	$\tau$-window spectral gap by $
		\Delta_\lambda
		\define
		1-\lambda
		\in(0,1]$.
\end{assumption}

Assumption~\ref{assumption:tv_gt_network} is a uniform window-mixing
condition: contraction is required only over products of $\tau$
consecutive mixing matrices, while an individual matrix $W_k$ need not
contract disagreement. Consequently, the communication graph may be
disconnected at individual iterations. The parameters $\tau$ and
$\Delta_\lambda$ quantify the effect of the time-varying network on the
convergence rate: smaller $\tau$ and larger $\Delta_\lambda$ correspond
to faster information mixing.

\begin{remark}[\bf\small Aligned-window contraction] \rm
	\label{remark:tv_gt_aligned}
	Condition~\eqref{tv_gt_window_contraction} requires contraction for
	windows starting at every iteration. A closely related formulation
	imposes contraction only on aligned blocks $
	\{\ell\tau_0,\ldots,(\ell+1)\tau_0-1\}$, $\ell\geq0$,
	as in the deterministic setting considered in
	\cite{koloskova2020unified}. Every window of length $2\tau_0-1$
	contains at least one complete aligned $\tau_0$-block. Since the
	remaining doubly stochastic mixing factors are nonexpansive, aligned
	$\tau_0$-window contraction with factor $\lambda<1$ implies
	\eqref{tv_gt_window_contraction} with $
	(\tau,\lambda)
	=
	(2\tau_0-1,\lambda)$;
	see also \cite{nedic2017achieving}.
\end{remark}

We let $\cF^k$ denote the information available immediately before the
stochastic samples $\{\xi_i^k\}_{i=1}^N$ are drawn at iteration $k$; in
particular, $\cF^k$ contains the initialization and all iterates and
stochastic samples generated up to iteration $k-1$.
\begin{assumption}[\bf\small Gradient noise] \rm
	\label{assumption:tv_gt_noise}
	Conditioned on $\cF^k$, the samples
	$\xi_1^k,\ldots,\xi_N^k$ are mutually independent and, for every
	$i$, satisfy
	\begin{subequations}
		\begin{align}
			\Ex\left[
			\grad F_i(x_i^k;\xi_i^k)
			-\grad f_i(x_i^k)
			\mid\cF^k
			\right]
			&=0
			\\
			\Ex\left[
			\left\|
			\grad F_i(x_i^k;\xi_i^k)
			-\grad f_i(x_i^k)
			\right\|^2
			\mid\cF^k
			\right]
			&\leq\sigma^2.
		\end{align}
	\end{subequations}
	The initial iterates $\{x_i^0\}_{i=1}^N$ are deterministic.
\end{assumption}

\subsection{Network form}

\noindent{\bfseries\small Network notation.}
For the analysis, we write \eqref{tv_gt_algorithm_agent} in stacked form.
Define
\begin{subequations}
	\label{def_network_x_grad}
	\begin{align}
		\x
		&\define
		\col\{x_1,\ldots,x_N\}
		\in\real^{Nd}
		\\
		\g
		&\define
		\col\{g_1,\ldots,g_N\}
		\in\real^{Nd}
		\\
		\grad\F(\x;\bxi)
		&\define
		\col\{\grad F_i(x_i;\xi_i)\}_{i=1}^N
		\in\real^{Nd}
		\\
		\grad\f(\x)
		&\define
		\col\{\grad f_i(x_i)\}_{i=1}^N
		\in\real^{Nd} \\
				\W_k
		&\define
		W_k\otimes I_d.
		\label{def_network_W}
	\end{align}
\end{subequations}

\noindent{\bfseries\small TV-GT algorithm.}
With the notation above, \eqref{tv_gt_algorithm_agent} becomes
\begin{subequations}
	\label{tv_gt_algorithm}
	\begin{align}
		\x^{k+1}
		&=
		\W_k(\x^k-\alpha\g^k)
		\label{tv_gt_algorithm_x}
		\\
		\g^{k+1}
		&=
		\W_k\g^k
		+
		\grad\F(\x^{k+1};\bxi^{k+1})
		-
		\grad\F(\x^k;\bxi^k),
		\label{tv_gt_algorithm_g}
	\end{align}
\end{subequations}
with $
\g^0=\grad\F(\x^0;\bxi^0)$
and arbitrary $\x^0$. We now separate the network dynamics into the evolution of the network
centroid and the disagreement among the agents.

\subsection{Centroid and disagreement}
\label{sec:tv_gt_centroid_disagreement}

\noindent{\bfseries\small Change of coordinates.}
Define
\begin{equation}
	\label{tv_gt_y_definition}
	\y^k
	\define
	\g^k-\grad\F(\x^k;\bxi^k).
\end{equation}
The initialization of $\g^0$ implies $\y^0=\zero$. This change of
coordinates removes the difference of consecutive stochastic gradients
from the tracking recursion. Substituting
\eqref{tv_gt_y_definition} into \eqref{tv_gt_algorithm} gives
\begin{subequations}
	\label{tv_gt_coordinates}
	\begin{align}
		\x^{k+1}
		&=
		\W_k
		(
		\x^k
		-\alpha\grad\F(\x^k;\bxi^k)
		-\alpha\y^k
		)
		\label{tv_gt_coordinates_x}
		\\
		\y^{k+1}
		&=
		\W_k\y^k
		-
		(\I-\W_k)
		\grad\F(\x^k;\bxi^k).
		\label{tv_gt_coordinates_y}
	\end{align}
\end{subequations}
For the centroid and disagreement analysis, define
\begin{subequations}
	\label{tv_gt_error_definitions}
	\begin{align}
				\mathbf\Pi
		&\define
		\Pi\otimes I_d
		=
		\left(\frac1N\one\one^\top\right)\otimes I_d
		\\
		\widetilde{\W}_k
		&\define
		\W_k-\mathbf\Pi
		=
		\widetilde W_k\otimes I_d \\
		\bar x^k
		&\define
		\frac1N(\one^\top\otimes I_d)\x^k
		=
		\frac1N\sum_{i=1}^N x_i^k
		\\
		\overline{\grad\f}(\x^k)
		&\define
		\frac1N\sum_{i=1}^N\grad f_i(x_i^k)
		\label{tv_gt_average_gradient}
		\\
		\widehat{\x}^k
		&\define
		(\I-\mathbf\Pi)\x^k
		\label{tv_gt_xhat_definition}
		\\
		\u^k
		&\define
		\grad\f(\x^k)-\grad\f(\x^\star),
		\quad
		\x^\star\define\one\otimes x^\star
		\label{tv_gt_u_definition}
		\\
		\widetilde{\y}^k
		&\define
		\y^k+\grad\f(\x^\star)
		\label{tv_gt_ytilde_definition} \\
			\label{grad_noise}
		\v^k
		&\define
		\grad\F(\x^k;\bxi^k)-\grad\f(\x^k),
		\quad
		\bar v^k
		\define
		\frac1N(\one^\top\otimes I_d)\v^k.
	\end{align}
\end{subequations}

\begin{lemma}[\bf\small Centroid and disagreement dynamics]
	\label{lemma:tv_gt_error_dynamics} \rm
	The centroid evolves as
	\begin{equation}
		\label{tv_gt_centroid}
		\bar x^{k+1}
		=
		\bar x^k
		-\alpha\overline{\grad\f}(\x^k)
		-\alpha\bar v^k,
	\end{equation}
	while the disagreement satisfies
	\begin{subequations}
		\label{tv_gt_centered}
		\begin{align}
			\widehat{\x}^{k+1}
			&=
			\widetilde{\W}_k
			(
			\widehat{\x}^k
			-\alpha\widetilde{\y}^k
			-\alpha\u^k
			-\alpha\v^k
			)
			\label{tv_gt_centered_x_primal}
			\\
			\widetilde{\y}^{k+1}
			&=
			\widetilde{\W}_k\widetilde{\y}^k
			-
			(\I-\W_k)(\u^k+\v^k).
			\label{tv_gt_centered_x_dual}
		\end{align}
	\end{subequations}
\end{lemma}

\begin{proof}
	By double stochasticity, $
	\frac1N(\one^\top\otimes I_d)\W_k
	=
	\frac1N(\one^\top\otimes I_d)$.
	Averaging \eqref{tv_gt_coordinates_y} and using $\y^0=\zero$ therefore
	gives $
	\frac1N(\one^\top\otimes I_d)\y^k=0$,
	$k\geq0$. 	Therefore, averaging \eqref{tv_gt_coordinates_x} and using
	\eqref{grad_noise} gives \eqref{tv_gt_centroid}.
	
	For the disagreement dynamics, note that $x^\star$ minimizes $f$,
	$\mathbf\Pi\grad\f(\x^\star)=\zero$, and hence
	$\mathbf\Pi\widetilde{\y}^k=\zero$. 	Moreover, $\grad\F(\x^k;\bxi^k)
	=\grad\f(\x^\star)+\u^k+\v^k$.
	Applying $\I-\mathbf\Pi$ to \eqref{tv_gt_coordinates_x} and using
	$(\I-\mathbf\Pi)\W_k=\widetilde{\W}_k$ gives
	\eqref{tv_gt_centered_x_primal}. Finally, substituting
	$\y^k=\widetilde{\y}^k-\grad\f(\x^\star)$ into
	\eqref{tv_gt_coordinates_y} yields
	\begin{equation}
		\label{tv_gt_ytilde_derivation}
		\widetilde{\y}^{k+1}
		=
		\W_k\widetilde{\y}^k
		-
		(\I-\W_k)(\u^k+\v^k).
	\end{equation}
	Since $\mathbf\Pi\widetilde{\y}^k=\zero$,
	$\W_k\widetilde{\y}^k
	=\widetilde{\W}_k\widetilde{\y}^k$, which gives
	\eqref{tv_gt_centered_x_dual}.
\end{proof}
\noindent The centroid recursion \eqref{tv_gt_centroid} is independent of the
particular mixing matrix $W_k$; the effect of the communication network is
therefore captured through the coupled disagreement dynamics
\eqref{tv_gt_centered}.

\subsection{Joint disagreement dynamics}

To balance the primal and tracking disagreements, define the scaled tracking disagreement
\begin{equation}
	\label{tv_gt_z_eta_choice}
	\mathbf z^k
	\define
	\frac{\alpha}{\eta}\widetilde{\y}^k, \quad 	\eta
	\define
	\frac{\Delta_\lambda}{2\tau}.
\end{equation}
Using \eqref{tv_gt_centered}, the two disagreement variables evolve as
\begin{subequations}
	\label{tv_gt_scaled_disagreement}
	\begin{align}
		\widehat{\x}^{k+1}
		&=
		\widetilde{\W}_k
		\left(
		\widehat{\x}^k
		-\eta\mathbf z^k
		-\alpha\u^k
		-\alpha\v^k
		\right)
		\label{tv_gt_scaled_disagreement_x}
		\\
		\mathbf z^{k+1}
		&=
		\widetilde{\W}_k\mathbf z^k
		-
		\frac{\alpha}{\eta}
		(\I-\W_k)(\u^k+\v^k).
		\label{tv_gt_scaled_disagreement_z}
	\end{align}
\end{subequations}
The scaling in \eqref{tv_gt_z_eta_choice} is introduced solely for the analysis. Its specific choice ensures that, once the one-step Lyapunov dissipation coefficient $\delta=	\frac{1-\lambda^2}{\tau}$ is established below, we have $\eta\leq \delta/2$. This allows the coupling between the primal and tracking disagreements to be absorbed by the network dissipation in the joint Lyapunov recursion.

\section{One-step Lyapunov analysis}
\label{sec:tv_gt_one_step_lyapunov}
This section builds, from the window-mixing condition alone, a time-varying
quadratic norm in which the disagreement dynamics dissipate at every
iteration; the remainder of the GT analysis is carried out in that norm.

\subsection{Time-varying quadratic Lyapunov norm}

Define the state-transition matrix of the lifted disagreement mixing process by
\begin{equation}
	\label{tv_gt_W_transition}
	\mathbf{\Psi}(t,s)
	\define
	\begin{cases}
		\widetilde{\W}_{t-1}\widetilde{\W}_{t-2}\cdots\widetilde{\W}_{s},
		& t>s,
		\\[1mm]
		\I,
		& t=s.
	\end{cases}
\end{equation}
For each $k\geq0$, define
\begin{equation}
	\label{tv_gt_P_definition}
	\P_k
	\define
	\sum_{r=0}^{\infty}
	\mathbf{\Psi}(k+r,k)^{\top}\mathbf{\Psi}(k+r,k).
\end{equation}
The quadratic form induced by $\P_k$ measures the accumulated future
disagreement energy starting from time $k$. The window contraction ensures
that this series is finite and uniformly bounded.

\begin{lemma}[\bf\small Time-varying quadratic mixing norm] \rm
	\label{lemma:tv_gt_R}
	Under Assumption~\ref{assumption:tv_gt_network}, define
	\begin{equation}
		\label{tv_gt_R_delta_definition}
		\R_k
		\define
		\delta\P_k,
		\qquad
		\delta
		\define
		\frac{1-\lambda^2}{\tau}.
	\end{equation}
	Then, for every $k\geq0$,
	\begin{subequations}
		\label{tv_gt_R_properties}
		\begin{align}
			\delta\I
			&\preceq
			\R_k
			\preceq
			\I,
			\label{tv_gt_R_equivalence}
			\\
			\widetilde{\W}_{k}^{\top}
			\R_{k+1}
			\widetilde{\W}_{k}
			&=
			\R_k-\delta\I.
			\label{tv_gt_R_lyap}
		\end{align}
	\end{subequations}
	Moreover, with $\Delta_\lambda\define1-\lambda$,
	\begin{equation}
		\label{tv_gt_delta_gap}
		\frac{\Delta_\lambda}{\tau}
		\leq
		\delta
		=
		\frac{\Delta_\lambda(1+\lambda)}{\tau}
		\leq
		\frac{2\Delta_\lambda}{\tau},
		\qquad
		0<\delta\leq\frac1\tau\leq1.
	\end{equation}
\end{lemma}

\begin{proof}
	See Appendix~\ref{app:lemma:tv_gt_R}.
\end{proof}

Equation \eqref{tv_gt_R_lyap} implies that the time-varying quadratic Lyapunov
function $
V_k(\q)\define\norm{\q}_{\R_k}^2$ applied to the homogeneous disagreement dynamics
$\q^{k+1}=\widetilde{\W}_k\q^k$ gives the exact one-step decrease
\[
V_{k+1}(\q^{k+1})
=
V_k(\q^k)-\delta\norm{\q^k}^2.
\]
Thus, although $\widetilde{\W}_k$ need not contract disagreement in the
Euclidean norm at every iteration, it is strictly dissipative in the
time-varying Lyapunov norm: the window-scale mixing property reappears as a
per-iteration decrease.
\begin{remark}[\bf\small Converse Lyapunov construction] \rm
	\label{remark:tv_gt_converse_lyapunov}
	The construction of $\P_k$ is motivated by classical converse-Lyapunov
	arguments for stable linear time-varying systems and is closely related to
	the Lyapunov--Stein equation and converse Lyapunov theory
	\cite{stein1952some,anderson1981detectability,
		geiselhart2014alternative}. Related ideas have also appeared in
	{\em continuous-time} multi-agent systems
	\cite{liu2017adaptive,liu2025distributed}. Here, we apply this construction
	directly to the discrete-time disagreement process underlying stochastic
	gradient tracking.
\end{remark}
\begin{remark}[\bf\small Per-step contraction] \rm
	When $\tau=1$, Assumption~\ref{assumption:tv_gt_network} implies $
		\|\widetilde{\W}_k\|_2
		\leq
		\lambda
		<1$, for every $k\geq0$.
	Hence, the Euclidean norm itself yields the one-step decrease $
		\|\widetilde{\W}_k\q\|^2
		\leq
		\lambda^2\|\q\|^2
		=
		\|\q\|^2
		-
		(1-\lambda^2)\|\q\|^2$.
	Therefore, in this case, a time-varying quadratic norm is not needed to
	obtain a per-iteration contraction. The construction in
	Lemma~\ref{lemma:tv_gt_R} is useful for the genuinely window-based case
	$\tau>1$.
\end{remark}
\subsection{Coupled error recursions}
Using the scaled variable in \eqref{tv_gt_z_eta_choice}, define
\begin{subequations}
	\label{tv_gt_disagreement_energy}
	\begin{align}
		\mathcal E_k
		&\define
		\|\widehat{\x}^k\|_{\R_k}^2
		+2\|\mathbf z^k\|_{\R_k}^2
		\\
		\widetilde{\mathcal E}_k
		&\define
		\|\widehat{\x}^k\|^2
		+2\|\mathbf z^k\|^2,
	\end{align}
\end{subequations}
and let
\begin{subequations}
	\label{tv_gt_ab_definition}
	\begin{align}
		\widetilde X_k
		&\define
		\Ex\|\bar x^k-x^\star\|^2
		\\
		E_k
		&\define
		\frac1N\Ex[\mathcal E_k],
		\qquad
		\widetilde E_k
		\define
		\frac1N\Ex[\widetilde{\mathcal E}_k].
	\end{align}
\end{subequations}
By \eqref{tv_gt_R_equivalence},
\begin{equation}
	\label{tv_gt_energy_sandwich}
	\delta\widetilde{\mathcal E}_k
	\leq
	\mathcal E_k
	\leq
	\widetilde{\mathcal E}_k,
\end{equation}
and hence
\begin{equation}
	\label{tv_gt_bs_sandwich}
	E_k
	\leq
	\widetilde E_k
	\leq
	\frac{E_k}{\delta},
	\qquad
	\frac1N\Ex\|\widehat{\x}^k\|^2
	\leq
	\widetilde E_k.
\end{equation}
Finally, define
\begin{equation}
	\label{tv_gt_Q_definition}
	Q
	\define
	1+\frac{8}{\eta^2}
	=
	1+\frac{32\tau^2}{\Delta_\lambda^2}.
\end{equation}
\begin{lemma}[\bf\small Coupled centroid--disagreement error] \rm
	\label{lemma:tv_gt_coupled}
	Under Assumptions~\ref{assumption:tv_gt_functions}--\ref{assumption:tv_gt_noise},
	if $0<\alpha\leq\frac1{4L}$, then, for every $k\geq0$,
	\begin{subequations}
		\label{tv_gt_coupled_recursion}
		\begin{align}
			\widetilde{X}_{k+1}
			&\leq
			(1-\mu\alpha)\widetilde{X}_k
			+\frac{3\alpha L}{2}\widetilde E_k
			+\frac{\alpha^2\sigma^2}{N}
			\label{tv_gt_coupled_a}
			\\
			E_{k+1}
			&\leq
			E_k
			-\frac{\delta}{4}\widetilde E_k
			+\frac{10\alpha^2QL^2}{\delta}
			\bigl(\widetilde{X}_k+\widetilde E_k\bigr)
			+\alpha^2Q\sigma^2.
			\label{tv_gt_coupled_b}
		\end{align}
	\end{subequations}
\end{lemma}
\begin{proof}
	See Appendix~\ref{app:lemma:tv_gt_coupled}.
\end{proof}
What Lemma~\ref{lemma:tv_gt_coupled} exploits is that a single time-varying
matrix $\R_k$ can serve as the quadratic norm for both the primal
disagreement $\widehat{\x}^k$ and the scaled tracking disagreement
$\mathbf z^k$.
The scaling $\mathbf z^k=(\alpha/\eta)\widetilde{\y}^k$ makes the
primal--tracking coupling sufficiently small relative to the network
decrement. Through the identity \eqref{tv_gt_R_lyap}, the window-contraction assumption
shows up as a one-step decrease of the combined energy $E_k$.

\section{Convergence results}
\label{sec:tv_gt_convergence_results}
We now give our main convergence results.
\subsection{Strongly convex result}
\begin{theorem}[\bf\small One-step contraction] \rm
	\label{thm:tv_gt_coupled_contraction}
	Suppose Assumptions~\ref{assumption:tv_gt_functions}--%
	\ref{assumption:tv_gt_noise} hold with $\mu>0$. If
	\begin{equation}
		\label{tv_gt_coupled_stepsize}
		0<\alpha
		\leq
		\bar\alpha
		\define
		\frac{\delta}
		{\sqrt{320}\,L\sqrt{Q\kappa}},
	\end{equation}
	define $
	V_k
	\define
	\widetilde X_k+\omega E_k$,
	$\omega
	\define
	\frac{16\alpha L}{\delta}$.
	Then
	\begin{equation}
		\label{tv_gt_V_recursion}
		V_{k+1}
		\leq
		\left(1-\frac{\mu\alpha}{2}\right)V_k
		+\frac{\alpha^2\sigma^2}{N}
		+\frac{16\alpha^3LQ\sigma^2}{\delta}.
	\end{equation}
\end{theorem}
\begin{proof}
	See Appendix~\ref{app:thm:tv_gt_coupled_contraction}.
\end{proof}

The preceding one-step recursion yields the following finite-horizon guarantee.
\begin{corollary}[\bf\small Transient rate and linear speedup]
	\label{corr:tv_gt_main} \rm
	Suppose the conditions in Theorem \ref{thm:tv_gt_coupled_contraction} hold. Then, for every horizon
	$K\geq1$, there exists a horizon-dependent stepsize such that
	\begin{align}
		\Ex\|\bar x^K-x^\star\|^2
		&\leq
		\widetilde{\mathcal O}\left(
		\frac{\sigma^2}{NK}
		+
		\frac{\tau^3\sigma^2}
		{(1-\lambda)^3K^2}
		+
		C_0
		\exp
		\left(-
		\frac{c(1-\lambda)^2K}
		{\kappa^{3/2}\tau^2}
		\right)\right),
		\label{tv_gt_transient_rate_explicit}
	\end{align}
	where  $c>0$ is an absolute constant, $	C_0\define
	\widetilde{X}_0
	+
	\frac{16\bar\alpha L}{\delta N}
	\|\widehat{\x}^0\|^2
	+
	\frac{32\bar\alpha^3L}{\eta^2\delta N}
	\|\grad\f(\x^\star)\|^2 $ is independent of $K$, and $\widetilde{\mathcal O}(\cdot)$ hides logarithmic factors and constants
	depending on $\mu$ and $L$. Moreover, linear speedup is achieved
	\begin{equation}
		\label{tv_gt_linear_speedup_rate}
		\Ex\|\bar x^K-x^\star\|^2
		\leq
		\widetilde{\mathcal O}\left(
		\frac{\sigma^2}{NK}
		\right)
	\end{equation}
	after transient time
		\begin{align}
		K_{\mathrm{transient}}
		=
		\widetilde{\mathcal O}\left(
		\frac{\kappa^{3/2}\tau^2}{(1-\lambda)^2}
		+
		\frac{N\kappa\tau^3}{(1-\lambda)^3}
		\right).
		\label{tv_gt_total_transient}
	\end{align}
\end{corollary}
\begin{proof}
	See Appendix~\ref{app:corr:tv_gt_main}.
\end{proof}
The leading term in \eqref{tv_gt_transient_rate_explicit} matches the
centralized stochastic rate obtained by averaging $N$ independent stochastic
gradients, and therefore gives linear speedup in the number of agents. The
network affects only the higher-order term and the transient period through
the window length $\tau$ and the window spectral gap $1-\lambda$. Thus,
individual graphs may be disconnected, as permitted by
Assumption~\ref{assumption:tv_gt_network}, without changing the leading
statistical term. Remark~\ref{remark:tv_gt_aligned} shows that the same
conclusion applies to the standard aligned-window model after a constant-factor
increase in the window length. When $\sigma=0$, Theorem~\ref{thm:tv_gt_coupled_contraction} with a constant
admissible stepsize gives geometric convergence to the exact solution.

\begin{corollary}[\bf\small Consensus error] \rm
	\label{corr:tv_gt_last_iterate_consensus}
	Let the conditions of Theorem~\ref{thm:tv_gt_coupled_contraction} hold,
	and let $\alpha=\alpha_K$ be the horizon-dependent stepsize
	\eqref{tv_gt_alpha_K} used in Corollary~\ref{corr:tv_gt_main}. Then, for
	every $K\geq1$,
	\begin{equation}
		\label{tv_gt_last_iterate_consensus_bound}
		\frac1N\Ex\|\widehat{\x}^K\|_{\R_K}^2
		\leq
		\widetilde{\mathcal O}\left(
		\frac{\tau^3\sigma^2}
		{(1-\lambda)^3K^2}
		+
		\frac{\tau^2E_0}
		{(1-\lambda)^2K^2}
		+
		C_0K
		\exp\left(
		-\frac{c\,(1-\lambda)^2K}{\kappa^{3/2}\tau^2}
		\right)
		\right),
	\end{equation}
	where $c>0$ is an absolute constant.
\end{corollary}
\begin{proof}
	See Appendix~\ref{app:tv_gt_last_iterate_consensus}.
\end{proof}	
\subsection{Convex result}

For the convex case ($\mu=0$), define the expected function gap
\begin{equation}
	\label{tv_gt_cvx_h_definition}
	\widetilde F_k
	\define
	\Ex[f(\bar x^k)-f(x^\star)].
\end{equation}
The following lemma gives the convex counterpart of
Lemma~\ref{lemma:tv_gt_coupled}.

\begin{lemma}[\bf\small Convex centroid--disagreement recursion] \rm
	\label{lemma:tv_gt_convex_coupled}
	Suppose Assumptions~\ref{assumption:tv_gt_functions}--%
	\ref{assumption:tv_gt_noise} hold with $\mu=0$. If
	$0<\alpha\leq1/(4L)$, then, for every $k\geq0$,
	\begin{subequations}
		\label{tv_gt_cvx_coupled_recursion}
		\begin{align}
			\widetilde X_{k+1}
			&\leq
			\widetilde X_k
			-\alpha\widetilde F_k
			+\frac{3\alpha L}{2}\widetilde E_k
			+\frac{\alpha^2\sigma^2}{N},
			\label{tv_gt_cvx_centroid_recursion}
			\\
			E_{k+1}
			&\leq
			E_k
			-\frac{\delta}{4}\widetilde E_k
			+\frac{20\alpha^2QL}{\delta}\widetilde F_k
			+\frac{10\alpha^2QL^2}{\delta}\widetilde E_k
			+\alpha^2Q\sigma^2.
			\label{tv_gt_cvx_disagreement_recursion}
		\end{align}
	\end{subequations}
\end{lemma}

\begin{proof}
	See Appendix~\ref{app:tv_gt_convex_coupled}.
\end{proof}

\begin{theorem}[\bf\small Convex finite-horizon bound] \rm
	\label{thm:tv_gt_convex}
	Suppose Assumptions~\ref{assumption:tv_gt_functions}--%
	\ref{assumption:tv_gt_noise} hold with $\mu=0$. If
	\begin{equation}
		\label{tv_gt_cvx_alpha_bnd}
		0<\alpha
		\leq
		\alpha_{\rm cvx}
		\define
		\frac{\delta}{\sqrt{480}\,L\sqrt Q},
	\end{equation}
	then, for every $K\geq1$,
	\begin{align}
		&\frac1K\sum_{k=0}^{K-1}
		\left(
		\Ex[f(\bar x^k)-f(x^\star)]
		+
		\frac{L}{N}\Ex\|\widehat{\x}^k\|^2
		\right)
		\nonumber\\
		&\qquad\leq
		\frac{8\widetilde X_0}{3\alpha K}
		+
		\frac{40LE_0}{\delta K}
		+
		\frac{8\alpha\sigma^2}{3N}
		+
		\frac{40\alpha^2QL\sigma^2}{\delta}.
		\label{tv_gt_cvx_finite_horizon}
	\end{align}
\end{theorem}

\begin{proof}
	See Appendix~\ref{app:tv_gt_convex_thm}.
\end{proof}

\begin{corollary}[\bf\small Convex rate and linear speedup] \rm
	\label{corr:tv_gt_convex_rate}
	Suppose the conditions of Theorem~\ref{thm:tv_gt_convex} hold. Then there exists a
	horizon-dependent stepsize such that
	\begin{align}
		&\frac1K\sum_{k=0}^{K-1}
		\left(
		\Ex[f(\bar x^k)-f(x^\star)]
		+
		\frac{L}{N}\Ex\|\widehat{\x}^k\|^2
		\right)
		\nonumber\\
		&\quad=
		\mathcal O\left(
		\frac{\sigma}{\sqrt{NK}}
		+
		\frac{\tau}{1-\lambda}
		\left(\frac{\sigma}{K}\right)^{2/3}
		+
		\frac{\tau^2\widetilde X_0}
		{(1-\lambda)^2K}
		+
		\frac{\tau\bar E_0}
		{(1-\lambda)K}
		\right),
		\label{tv_gt_cvx_network_rate}
	\end{align}
	where $
	\bar E_0
	\define
	\frac1N
	\left(
	\|\widehat{\x}^0\|_{\R_0}^2
	+
	\frac{2\alpha_{\rm cvx}^2}{\eta^2}
	\|\grad\f(\x^\star)\|_{\R_0}^2
	\right)$. 
\end{corollary}

\begin{proof}
	See Appendix~\ref{app:tv_gt_convex_rate}.
\end{proof}
The leading term in \eqref{tv_gt_cvx_network_rate} scales as
$\mathcal O(\sigma/\sqrt{NK})$ and is independent of the network parameters,
matching the centralized mini-batch stochastic-gradient rate. The remaining
terms decay faster with $K$ and capture the effect of the time-varying
network. In particular, comparing each network-dependent term with the
leading statistical term shows that they become lower order after $
	K_{\mathrm{transient}}
	=
	\mathcal O\left(
	\frac{N^3\tau^6}
	{(1-\lambda)^6}
	\right)$.
Hence, beyond this network-dependent transient, the convergence rate is
dominated by the centralized term
$\mathcal O(\sigma/\sqrt{NK})$, and the method achieves
linear speedup in the number of agents despite the time-varying network. In the noiseless case, setting $\sigma=0$ in \eqref{tv_gt_cvx_finite_horizon} and using a constant admissible stepsize
gives $\mathcal O(\tau^2/((1-\lambda)^2K))$.

\begin{remark}[\bf\small Consensual initialization] \rm
	\label{remark:tv_gt_convex_consensual_initialization}
	If $x_i^0=x^0$ for every $i$, then
	$\widehat{\x}^0=\zero$ and
	$\widetilde X_0=\|x^0-x^\star\|^2$. Moreover, $
	\bar E_0
	=
	\frac{2\alpha_{\rm cvx}^2}{\eta^2N}
	\|\grad\f(\x^\star)\|_{\R_0}^2$. 	Thus, the local-gradient heterogeneity at the solution enters only
	through the $K^{-1}$ initialization term in
	\eqref{tv_gt_cvx_network_rate}.
\end{remark}
\section{Conclusion and future directions}
\label{sec:tv_gt_conclusion}

This paper analyzed stochastic gradient tracking over deterministic
time-varying networks under a window-mixing condition that permits
individual communication graphs to be disconnected. The time-varying
quadratic Lyapunov construction converts contraction over communication
windows into a one-step decrease, leading to finite-time guarantees for both
strongly convex and convex objectives. In both cases, the dominant
statistical term matches the corresponding centralized mini-batch rate after
a network-dependent transient.

Future work includes extending the one-step Lyapunov framework to random communication graphs, smooth nonconvex objectives, and time-varying directed push--pull methods. In particular, whether time-varying push--pull admits linear speedup is open; even for static networks such guarantees appeared only
recently~\cite{liang2025linear}.

\newpage
\appendices
\section{Auxiliary results}
\label{appendix:auxiliary_results}

We first collect the basic mixing and stochastic-noise consequences of the
main assumptions that are used in the proofs. We also recall a standard bound
for the centroid error.
\begin{proposition}[\bf\small Properties of the mixing matrices]
	\label{prop:tv_gt_mixing_properties} \rm
	Under Assumption~\ref{assumption:tv_gt_network}, for every $k\geq0$ and
	every $t\geq s\geq0$,
	\begin{subequations}
		\begin{align}
			\|\widetilde{\W}_{k+\tau-1:k}\|_2
			&\leq
			\lambda,
			\qquad k\geq0
			\label{tv_gt_lifted_window_contraction}
			\\
			\|\widetilde{\W}_{t:s}\|_2
			&\leq
			1
			\label{tv_gt_product_identity}
			\\
			\|\I-\W_k\|_2
			&\leq
			2,
			\label{tv_gt_matrix_norms}
		\end{align}
	\end{subequations}
	where $	\widetilde{\W}_{t:s}
	\define
	\widetilde W_{t:s}\otimes I_d
	=
	\W_t\W_{t-1}\cdots\W_s-\mathbf\Pi$.
\end{proposition}

\begin{proof}
	Since $W_k$ is nonnegative and doubly stochastic,
	$\|W_k\|_1=\|W_k\|_\infty=1$. Hence
	\begin{equation}
		\|W_k\|_2
		\leq
		\sqrt{\|W_k\|_1\|W_k\|_\infty}
		=
		1.
	\end{equation}
	Moreover, double stochasticity implies
	$\Pi W_k=W_k\Pi=\Pi$. Therefore,
	\begin{align}
		\widetilde W_{t:s}
		&=
		(W_t-\Pi)(W_{t-1}-\Pi)\cdots(W_s-\Pi)
		\nonumber\\
		&=
		W_tW_{t-1}\cdots W_s-\Pi.
	\end{align}
	Since the product $W_t\cdots W_s$ is also nonnegative and doubly
	stochastic,
	\begin{equation}
		\|\widetilde W_{t:s}\|_2
		=
		\|(I_N-\Pi)W_t\cdots W_s\|_2
		\leq
		1.
	\end{equation}
	The bound
	$\|I_N-W_k\|_2\leq2$
	then follows from the triangle inequality.	For the lifted matrices,
	$\widetilde{\W}_k=\widetilde W_k\otimes I_d$, and hence $\widetilde{\W}_{t:s}
		=
		\widetilde W_{t:s}\otimes I_d$.
	Using $\|A\otimes I_d\|_2=\|A\|_2$, all the corresponding lifted
	bounds follow immediately from the scalar-matrix bounds and
	Assumption~\ref{assumption:tv_gt_network}.
\end{proof}

\begin{proposition}[\bf\small Gradient-noise bounds]
	\label{prop:tv_gt_noise_bounds} \rm
	Under Assumption~\ref{assumption:tv_gt_noise}, for every $k\geq0$,
	\begin{equation}
		\label{tv_gt_noise_bounds}
		\Ex\bigl[\|\v^k\|^2\mid\cF^k\bigr]
		\leq
		N\sigma^2,
		\qquad
		\Ex\bigl[\|\bar v^k\|^2\mid\cF^k\bigr]
		\leq
		\frac{\sigma^2}{N}.
	\end{equation}
\end{proposition}

\begin{proof}
	Write
	\begin{equation}
		v_i^k
		\define
		\grad F_i(x_i^k;\xi_i^k)-\grad f_i(x_i^k),
	\end{equation}
	so that
	$\v^k=\col\{v_1^k,\ldots,v_N^k\}$.
	By conditional unbiasedness and the bounded-variance assumption, $
	\Ex[v_i^k\mid\cF^k]=0$ and $\Ex[\|v_i^k\|^2\mid\cF^k]
	\leq
	\sigma^2$.
	Therefore,
	\begin{align}
		\Ex[\|\v^k\|^2\mid\cF^k]
		&=
		\sum_{i=1}^N
		\Ex[\|v_i^k\|^2\mid\cF^k]
		\leq
		N\sigma^2.
	\end{align}
	Moreover, for $\bar v^k
	=
	\frac1N\sum_{i=1}^N v_i^k$, conditional independence and zero conditional means imply that the cross
	terms vanish, and hence
	\begin{align}
		\Ex[\|\bar v^k\|^2\mid\cF^k]
		&=
		\frac1{N^2}
		\sum_{i=1}^N
		\Ex[\|v_i^k\|^2\mid\cF^k]
		\leq
		\frac{\sigma^2}{N}.
	\end{align}
\end{proof}
The following lemma provides the basic bound for the centroid error.
\begin{lemma}[\bf\small Centroid error inequality]
	\label{lemma:tv_gt_average} \rm
	Under Assumptions~\ref{assumption:tv_gt_functions}--\ref{assumption:tv_gt_noise},
	if $0<\alpha\leq\frac1{4L}$,
	then, for every $k\geq0$,
	\begin{align}
		\Ex\|\bar x^{k+1}-x^\star\|^2
		&\leq
		(1-\mu\alpha)
		\Ex\|\bar x^k-x^\star\|^2
		-\alpha
		\Ex[
		f(\bar x^k)-f(x^\star)]
		\nonumber\\
		&\quad
		+\frac{3\alpha L}{2N}
		\Ex\|\widehat{\x}^k\|^2
		+\frac{\alpha^2\sigma^2}{N}.
		\label{centroid_ineq}
	\end{align}
\end{lemma}
\begin{proof}
The proof is standard; see, \eg, \cite[Lemma 3]{alghunaim2023enhanced}.
	Conditioning on $\cF^k$ and using
	$\Ex[\bar v^k\mid\cF^k]=0$ in \eqref{tv_gt_centroid} gives
	\begin{align}
		\Ex\!\left[
		\|\bar x^{k+1}-x^\star\|^2
		\,\middle|\,
		\cF^k
		\right]
		&=
		\left\|
		\bar x^k-x^\star
		-\alpha\overline{\grad\f}(\x^k)
		\right\|^2
		+
		\alpha^2
		\Ex\!\left[
		\|\bar v^k\|^2
		\,\middle|\,
		\cF^k
		\right]
		\nonumber\\
		&\leq
		\|\bar x^k-x^\star\|^2
		-2\alpha
		\left\langle
		\bar x^k-x^\star,
		\overline{\grad\f}(\x^k)
		\right\rangle
		\nonumber\\
		&\quad
		+\alpha^2
		\|\overline{\grad\f}(\x^k)\|^2
		+\frac{\alpha^2\sigma^2}{N},
		\label{tv_gt_centroid_conditional}
	\end{align}
	where the last inequality follows from
	\eqref{tv_gt_noise_bounds}.	We first bound the inner-product term. For each agent $i$, $\mu$-strong convexity of $f_i$ implies:
	\begin{equation}
		f_i(x^\star)
		\geq
		f_i(x_i^k)
		+
		\left\langle
		\grad f_i(x_i^k),
		x^\star-x_i^k
		\right\rangle
		+
		\frac{\mu}{2}
		\|x_i^k-x^\star\|^2.
	\end{equation}
	Rearranging gives
	\begin{equation}
		\label{tv_gt_strong_convexity_agent}
		\left\langle
		x_i^k-x^\star,
		\grad f_i(x_i^k)
		\right\rangle
		\geq
		f_i(x_i^k)-f_i(x^\star)
		+
		\frac{\mu}{2}
		\|x_i^k-x^\star\|^2.
	\end{equation}
	Using $
	\bar x^k-x^\star
	=
	x_i^k-x^\star
	+
	\bar x^k-x_i^k$, 
	we obtain
	\begin{align}
		\left\langle
		\bar x^k-x^\star,
		\grad f_i(x_i^k)
		\right\rangle
		&=
		\left\langle
		x_i^k-x^\star,
		\grad f_i(x_i^k)
		\right\rangle
		+
		\left\langle
		\bar x^k-x_i^k,
		\grad f_i(x_i^k)
		\right\rangle
		\nonumber\\
		&\geq
		f_i(x_i^k)-f_i(x^\star)
		+
		\left\langle
		\grad f_i(x_i^k),
		\bar x^k-x_i^k
		\right\rangle
		+
		\frac{\mu}{2}
		\|x_i^k-x^\star\|^2.
		\label{tv_gt_inner_product_agent_pre}
	\end{align}
	On the other hand, the $L$-smoothness of $f_i$ gives:
	\begin{equation}
		f_i(\bar x^k)
		\leq
		f_i(x_i^k)
		+
		\left\langle
		\grad f_i(x_i^k),
		\bar x^k-x_i^k
		\right\rangle
		+
		\frac{L}{2}
		\|\bar x^k-x_i^k\|^2.
	\end{equation}
	Equivalently,
	\begin{equation}
		f_i(x_i^k)
		+
		\left\langle
		\grad f_i(x_i^k),
		\bar x^k-x_i^k
		\right\rangle
		\geq
		f_i(\bar x^k)
		-
		\frac{L}{2}
		\|x_i^k-\bar x^k\|^2.
	\end{equation}
	Substituting this bound into
	\eqref{tv_gt_inner_product_agent_pre} yields
	\begin{align}
		\left\langle
		\bar x^k-x^\star,
		\grad f_i(x_i^k)
		\right\rangle
		&\geq
		f_i(\bar x^k)-f_i(x^\star)
		+
		\frac{\mu}{2}
		\|x_i^k-x^\star\|^2
		-
		\frac{L}{2}
		\|x_i^k-\bar x^k\|^2.
		\label{tv_gt_inner_product_agent}
	\end{align}	
	Averaging \eqref{tv_gt_inner_product_agent} over
	$i=1,\ldots,N$ and recalling $
	\overline{\grad\f}(\x^k)
	=
	\frac1N\sum_{i=1}^N\grad f_i(x_i^k)$, 
	gives
	\begin{align}
		\left\langle
		\bar x^k-x^\star,
		\overline{\grad\f}(\x^k)
		\right\rangle
		&\geq
		f(\bar x^k)-f(x^\star)
		+
		\frac{\mu}{2N}
		\sum_{i=1}^N
		\|x_i^k-x^\star\|^2
		-
		\frac{L}{2N}
		\sum_{i=1}^N
		\|x_i^k-\bar x^k\|^2.
		\label{tv_gt_inner_product_average_pre}
	\end{align}
	The standard variance decomposition gives
	\begin{align}
		\sum_{i=1}^N
		\|x_i^k-x^\star\|^2
		&=
		\sum_{i=1}^N
		\|x_i^k-\bar x^k\|^2
		+
		N\|\bar x^k-x^\star\|^2
		\nonumber \\
		&\geq
		N\|\bar x^k-x^\star\|^2,
		\label{tv_gt_variance_decomposition}
	\end{align}
	where we used
	$\sum_{i=1}^N(x_i^k-\bar x^k)=0$.
	Applying these relations to
	\eqref{tv_gt_inner_product_average_pre} gives
	\begin{equation}
		\label{tv_gt_inner_product_bound}
		\left\langle
		\bar x^k-x^\star,
		\overline{\grad\f}(\x^k)
		\right\rangle
		\geq
		f(\bar x^k)-f(x^\star)
		+
		\frac{\mu}{2}
		\|\bar x^k-x^\star\|^2
		-
		\frac{L}{2N}
		\|\widehat{\x}^k\|^2,
	\end{equation}
	where we used $\|\widehat{\x}^k\|^2=\sum_{i=1}^N
	\|x_i^k-\bar x^k\|^2$.
	Next, decompose the average local gradient as
	\begin{equation}
		\overline{\grad\f}(\x^k)
		=
		\grad f(\bar x^k)
		+
		\frac1N
		\sum_{i=1}^N
		\left(
		\grad f_i(x_i^k)-\grad f_i(\bar x^k)
		\right).
	\end{equation}
	Using $\|a+b\|^2\leq2\|a\|^2+2\|b\|^2$, Jensen's inequality, and
	$L$-smoothness,
	\begin{align}
		\|\overline{\grad\f}(\x^k)\|^2
		&\leq
		2\|\grad f(\bar x^k)\|^2
		+
		\frac{2}{N}
		\sum_{i=1}^N
		\|\grad f_i(x_i^k)-\grad f_i(\bar x^k)\|^2
		\nonumber\\
		&\leq
		4L\bigl(f(\bar x^k)-f(x^\star)\bigr)
		+
		\frac{2L^2}{N}
		\|\widehat{\x}^k\|^2,
		\label{tv_gt_average_gradient_bound}
	\end{align}
	where we also used $
	\|\grad f(\bar x^k)\|^2
	\leq
	2L(f(\bar x^k)-f(x^\star))$. 
	Substituting \eqref{tv_gt_inner_product_bound} and
	\eqref{tv_gt_average_gradient_bound} into
	\eqref{tv_gt_centroid_conditional} gives
	\begin{align}
		\Ex\!\left[
		\|\bar x^{k+1}-x^\star\|^2
		\,\middle|\,
		\cF^k
		\right]
		&\leq
		(1-\mu\alpha)
		\|\bar x^k-x^\star\|^2
		-\bigl(2\alpha-4\alpha^2L\bigr)
		\bigl(f(\bar x^k)-f(x^\star)\bigr)
		\nonumber\\
		&\quad
		+\frac{\alpha L+2\alpha^2L^2}{N}
		\|\widehat{\x}^k\|^2
		+\frac{\alpha^2\sigma^2}{N}.
		\label{tv_gt_centroid_pre_final}
	\end{align}
	Since $\alpha\leq1/(4L)$,
	\begin{equation}
		2\alpha-4\alpha^2L
		\geq
		\alpha,
		\qquad
		\alpha L+2\alpha^2L^2
		\leq
		\frac{3\alpha L}{2}.
	\end{equation}
	Applying these bounds in \eqref{tv_gt_centroid_pre_final} and taking total
	expectation proves \eqref{centroid_ineq}.
\end{proof}

\section{Proofs of the main results}
\label{appendix:main_proofs}

\subsection{Proof of Lemma~\ref{lemma:tv_gt_R}}
\label{app:lemma:tv_gt_R}
\begin{proof}
	Write $r=\ell\tau+t$, where $\ell,t$ are integers with $\ell\geq0$ and
	$0\leq t<\tau$. The transition $\mathbf{\Psi}(k+r,k)$ defined in \eqref{tv_gt_W_transition} can then be decomposed
	into $\ell$ complete blocks of length $\tau$ and one remaining block of
	length $t<\tau$. By Assumption~\ref{assumption:tv_gt_network}, every complete
	$\tau$-block contracts disagreement by at most $\lambda$, whereas the
	remaining block is nonexpansive since each doubly stochastic mixing
	matrix satisfies $\|\widetilde{\W}_j\|_2\leq1$. Hence,
	\begin{equation}
		\label{tv_gt_Psi_window_bound}
		\|\mathbf{\Psi}(k+r,k)\|_2
		\leq
		\lambda^\ell
		=
		\lambda^{\lfloor r/\tau\rfloor},
		\end{equation}
		where, when $\lambda=0$, the usual convention $\lambda^0=1$ is used. 
	Consequently,
	\begin{equation}
		\label{tv_gt_Psi_quadratic_bound}
		\mathbf{\Psi}(k+r,k)^\top\mathbf{\Psi}(k+r,k)
		\preceq
		\lambda^{2\lfloor r/\tau\rfloor}\I.
	\end{equation}
	Since the $r=0$ term in \eqref{tv_gt_P_definition} is $\I$, we have
	$\P_k\succeq\I$. Using \eqref{tv_gt_Psi_quadratic_bound} for the upper
	bound and writing each $r\geq0$ uniquely as $r=\ell\tau+t$ gives
	\begin{align}
		\I
		\preceq
		\P_k
		&=
		\sum_{r=0}^{\infty}
		\mathbf{\Psi}(k+r,k)^\top\mathbf{\Psi}(k+r,k)
		\nonumber\\
		&\preceq
		\sum_{r=0}^{\infty}
		\lambda^{2\lfloor r/\tau\rfloor}\I
		\nonumber\\
		&=
		\sum_{\ell=0}^{\infty}
		\sum_{t=0}^{\tau-1}
		\lambda^{2\ell}\I
		\nonumber\\
		&=
		\tau
		\sum_{\ell=0}^{\infty}
		\lambda^{2\ell}\I
		=
		\frac{\tau}{1-\lambda^2}\I
		=
		\frac1\delta\I.
		\label{tv_gt_P_bounds}
	\end{align}
	Multiplying \eqref{tv_gt_P_bounds} by $\delta$ and recalling
	$\R_k=\delta\P_k$ yields \eqref{tv_gt_R_equivalence}.

	Now we show \eqref{tv_gt_R_lyap}. For $r\geq1$,  by the definition in \eqref{tv_gt_W_transition}, the transition matrix can be factored as
	\begin{equation}
		\label{tv_gt_Psi_factorization}
		\mathbf{\Psi}(k+r,k)
		=
		\mathbf{\Psi}(k+r,k+1)\widetilde{\W}_k.
	\end{equation}
		Separating the $r=0$ term in \eqref{tv_gt_P_definition}, for which
	$\mathbf{\Psi}(k,k)=\I$, and using \eqref{tv_gt_Psi_factorization} gives
	\begin{align}
		\P_k
		&=
		\I
		+
		\sum_{r=1}^{\infty}
		\mathbf{\Psi}(k+r,k)^\top
		\mathbf{\Psi}(k+r,k)
		\nonumber\\
		&=
		\I
		+
		\widetilde{\W}_k^\top
		\left(
		\sum_{r=1}^{\infty}
		\mathbf{\Psi}(k+r,k+1)^\top
		\mathbf{\Psi}(k+r,k+1)
		\right)
		\widetilde{\W}_k.
		\label{tv_gt_P_recursion_intermediate}
	\end{align}
Introducing the change of index $r'=r-1$, the summation in
\eqref{tv_gt_P_recursion_intermediate} satisfies
\begin{align}
&	\sum_{r=1}^{\infty}
	\mathbf{\Psi}(k+r,k+1)^\top
	\mathbf{\Psi}(k+r,k+1)
\nonumber \\
	&=
	\sum_{r'=0}^{\infty}
	\mathbf{\Psi}(k+1+r',k+1)^\top
	\mathbf{\Psi}(k+1+r',k+1)
	=
	\P_{k+1}.
	\label{tv_gt_P_index_shift}
\end{align}
	Consequently,
	\begin{equation}
		\label{tv_gt_P_lyap_identity}
		\P_k
		=
		\I
		+
		\widetilde{\W}_k^\top
		\P_{k+1}
		\widetilde{\W}_k.
	\end{equation}
	Multiplying \eqref{tv_gt_P_lyap_identity} by $\delta$ and recalling
	$\R_k=\delta\P_k$ yields \eqref{tv_gt_R_lyap}.	Finally, from the definition of $\delta$,
	\begin{equation}
		\delta
		=
		\frac{1-\lambda^2}{\tau}
		=
		\frac{\Delta_\lambda(1+\lambda)}{\tau}.
		\label{tv_gt_delta_expansion}
	\end{equation}
	Since $\lambda\in[0,1)$, we have
	$1\leq1+\lambda<2$, and hence
	\begin{equation}
		\frac{\Delta_\lambda}{\tau}
		\leq
		\delta
		\leq
		\frac{2\Delta_\lambda}{\tau}.
	\end{equation}
	Moreover, $0<1-\lambda^2\leq1$ and $\tau\geq1$, so $
		0<\delta
		\leq
		\frac1\tau
		\leq
		1$.
	This proves \eqref{tv_gt_delta_gap}.
\end{proof}

\subsection{Proof of Lemma~\ref{lemma:tv_gt_coupled}}
\label{app:lemma:tv_gt_coupled}
\begin{proof}
	Equation~\eqref{tv_gt_coupled_a} follows from
	\eqref{centroid_ineq} and \eqref{tv_gt_bs_sandwich} after dropping the
	nonnegative function-gap term.

	We next establish the disagreement recursion by first recalling from \eqref{tv_gt_scaled_disagreement} that
	\begin{subequations}
		\label{tv_gt_scaled_disagreement_app}
		\begin{align}
			\widehat{\x}^{k+1}
			&=
			\widetilde{\W}_k
			\left(
			\widehat{\x}^k
			-\eta\mathbf z^k
			-\alpha\u^k
			-\alpha\v^k
			\right),
			\label{tv_gt_scaled_disagreement_x_app}
			\\
			\mathbf z^{k+1}
			&=
			\widetilde{\W}_k\mathbf z^k
			-
			\tfrac{\alpha}{\eta}
			(\I-\W_k)(\u^k+\v^k).
			\label{tv_gt_scaled_disagreement_z_app}
		\end{align}
	\end{subequations}
	 Define
	\begin{equation}
		\label{tv_gt_H_definition}
		\mathbf H_k
		\define
		\widetilde{\W}_k^\top\R_{k+1}\widetilde{\W}_k
		=
		\R_k-\delta\I.
	\end{equation}
	By \eqref{tv_gt_R_equivalence} and \eqref{tv_gt_R_lyap},
	$\zero\preceq\mathbf H_k\preceq\R_k\preceq\I$. Moreover,
	\eqref{tv_gt_z_eta_choice}, \eqref{tv_gt_R_delta_definition}, and \eqref{tv_gt_delta_gap} imply
	\begin{equation}
		\label{tv_gt_eta_delta_relation}
		\eta
		=
		\frac{\delta}{2(1+\lambda)}
		\leq
		\frac{\delta}{2}
		\leq
		\frac12.
	\end{equation}
	Consider first the homogeneous part of
	\eqref{tv_gt_scaled_disagreement_app}. Using \eqref{tv_gt_H_definition},
	its one-step energy is
	\begin{align}
		\mathcal E_{k+1}^{\rm h}
		&\define
		\|\widetilde{\W}_k(\widehat{\x}^k-\eta\mathbf z^k)\|_{\R_{k+1}}^2
		+2\|\widetilde{\W}_k\mathbf z^k\|_{\R_{k+1}}^2
		\nonumber\\
		&=
		\|\widehat{\x}^k-\eta\mathbf z^k\|_{\mathbf H_k}^2
		+2\|\mathbf z^k\|_{\mathbf H_k}^2
		\nonumber\\
		&=
		\mathcal E_k
		-\delta\|\widehat{\x}^k\|^2
		-2\delta\|\mathbf z^k\|^2
		-2\eta(\widehat{\x}^k)^\top\mathbf H_k\mathbf z^k
		+\eta^2\|\mathbf z^k\|_{\mathbf H_k}^2.
		\label{tv_gt_homogeneous_energy_exact}
	\end{align}
	Using
	$2|(\widehat{\x}^k)^\top\mathbf H_k\mathbf z^k|
	\leq\|\widehat{\x}^k\|_{\mathbf H_k}^2+\|\mathbf z^k\|_{\mathbf H_k}^2$, $\mathbf H_k\preceq\I$, and \eqref{tv_gt_eta_delta_relation} gives
	\begin{align}
		\mathcal E_{k+1}^{\rm h}
		&\leq
		\mathcal E_k
		-(\delta-\eta)\|\widehat{\x}^k\|^2
		-(2\delta-\eta-\eta^2)\|\mathbf z^k\|^2
		\nonumber\\
		&\leq
		\mathcal E_k
		-(\delta-\eta)(\|\widehat{\x}^k\|^2+2\|\mathbf z^k\|^2)		\nonumber\\
		&\leq
		\mathcal E_k
		-\frac{\delta}{2}\widetilde{\mathcal E}_k.
		\label{tv_gt_homogeneous_energy_bound}
	\end{align}
	The variables $\widehat{\x}^k$, $\mathbf z^k$, and $\u^k$ are
$\cF^k$-measurable, while
$\Ex[\v^k\mid\cF^k]=\zero$. Hence, conditioning on $\cF^k$ and
expanding the two terms in $\mathcal E_{k+1}$, the cross terms involving
$\v^k$ vanish, and we obtain
\begin{align}
	\Ex[\mathcal E_{k+1}\mid\cF^k]
	={}&
	\bigl\|
	\widetilde{\W}_k
	(\widehat{\x}^k-\eta\mathbf z^k-\alpha\u^k)
	\bigr\|_{\R_{k+1}}^2
	\nonumber\\
	&\quad
	+2\bigl\|
	\widetilde{\W}_k\mathbf z^k
	-\tfrac{\alpha}{\eta}(\I-\W_k)\u^k
	\bigr\|_{\R_{k+1}}^2
	\nonumber\\
	&\quad
	+\alpha^2
	\Ex\!\left[
	\|\widetilde{\W}_k\v^k\|_{\R_{k+1}}^2
	+\frac{2}{\eta^2}
	\|(\I-\W_k)\v^k\|_{\R_{k+1}}^2
	\,\middle|\,\cF^k
	\right].
	\label{tv_gt_network_conditional}
\end{align}
For any positive semidefinite matrix $\R$, vectors $a,b$, and $\beta>0$,
Young's inequality gives
\begin{equation}
	\|a+b\|_{\R}^2
	\leq
	(1+\beta)\|a\|_{\R}^2
	+
	\left(1+\frac{1}{\beta}\right)\|b\|_{\R}^2.
	\label{tv_gt_weighted_young}
\end{equation}
Applying \eqref{tv_gt_weighted_young} with $\beta=\delta/4$ to each of the
first two terms yields
\begin{align}
	\Ex[\mathcal E_{k+1}\mid\cF^k]
	\leq{}&
	\left(1+\frac{\delta}{4}\right)
	\mathcal E_{k+1}^{\rm h}
	\nonumber\\
	&+
	\left(1+\frac{4}{\delta}\right)\alpha^2
	\left(
	\|\widetilde{\W}_k\u^k\|_{\R_{k+1}}^2
	+\frac{2}{\eta^2}
	\|(\I-\W_k)\u^k\|_{\R_{k+1}}^2
	\right)
	\nonumber\\
	&+
	\alpha^2
	\Ex\!\left[
	\|\widetilde{\W}_k\v^k\|_{\R_{k+1}}^2
	+\frac{2}{\eta^2}
	\|(\I-\W_k)\v^k\|_{\R_{k+1}}^2
	\,\middle|\,\cF^k
	\right].
	\label{tv_gt_network_young}
\end{align}
Since $\R_{k+1}\preceq\I$,
$\|\widetilde{\W}_k\|_2\leq1$, and
$\|\I-\W_k\|_2\leq2$, for any vector $q$,
\begin{equation}
	\|\widetilde{\W}_k q\|_{\R_{k+1}}^2
	+\frac{2}{\eta^2}
	\|(\I-\W_k)q\|_{\R_{k+1}}^2
	\leq
	\left(1+\frac{8}{\eta^2}\right)\|q\|^2
	=
	Q\|q\|^2.
\end{equation}
Therefore, using \eqref{tv_gt_homogeneous_energy_bound} and
\eqref{tv_gt_noise_bounds},
\begin{align}
	\Ex[\mathcal E_{k+1}\mid\cF^k]
	&\leq
	\left(1+\frac{\delta}{4}\right)
	\left(
	\mathcal E_k-\frac{\delta}{2}
	\widetilde{\mathcal E}_k
	\right)
	+
	\left(1+\frac{4}{\delta}\right)
	\alpha^2Q\|\u^k\|^2
	+
	\alpha^2QN\sigma^2
	\nonumber\\
	&\leq
	\mathcal E_k
	-\frac{\delta}{4}\widetilde{\mathcal E}_k
	+
	\frac{5\alpha^2Q}{\delta}\|\u^k\|^2
	+
	\alpha^2QN\sigma^2.
	\label{tv_gt_network_pre_coupled}
\end{align}
In the last inequality, we used
$\mathcal E_k\leq\widetilde{\mathcal E}_k$ to bound
$(\delta/4)\mathcal E_k
\leq(\delta/4)\widetilde{\mathcal E}_k$,
discarded the nonpositive term
$-(\delta^2/8)\widetilde{\mathcal E}_k$,
and used $0<\delta\leq1$ to obtain
$1+4/\delta\leq5/\delta$. 	For the gradient gap, split each entry of $	\u^k=
\grad\f(\x^k)-\grad\f(\x^\star)$ at $\bar x^k$ and use
	$L$-smoothness:
	\begin{align}
		\|\u^k\|^2
		&\leq
		2\sum_{i=1}^{N}
		\|\grad f_i(x_i^k)-\grad f_i(\bar x^k)\|^2
		+2\sum_{i=1}^{N}
		\|\grad f_i(\bar x^k)-\grad f_i(x^\star)\|^2
		\nonumber\\
		&\leq
		2L^2\|\widehat{\x}^k\|^2
		+2NL^2\|\bar x^k-x^\star\|^2.
		\label{tv_gt_u_bound}
	\end{align}
	Taking expectations, dividing by $N$, and using
	\eqref{tv_gt_bs_sandwich} yields
	\begin{equation}
		\label{tv_gt_u_coupled}
		\frac1N\Ex\|\u^k\|^2
		\leq
		2L^2(\widetilde{X}_k+\widetilde E_k).
	\end{equation}
	Taking total expectation in \eqref{tv_gt_network_pre_coupled}, dividing
	by $N$, and substituting \eqref{tv_gt_u_coupled} gives
	\eqref{tv_gt_coupled_b}.
\end{proof}

\subsection{Proof of Theorem~\ref{thm:tv_gt_coupled_contraction}}
\label{app:thm:tv_gt_coupled_contraction}
\begin{proof}
		Since $\kappa=L/\mu$, the stepsize cap in \eqref{tv_gt_coupled_stepsize}
		can be written as
		\begin{equation}
			\bar\alpha
			=
			\frac{\delta}{\sqrt{320}\,L\sqrt{Q\kappa}}
			=
			\sqrt{\frac{\mu\delta^2}{320QL^3}}.
			\label{tv_gt_alpha_bar_forms}
		\end{equation}
		Since $\mu\leq L$, $0<\delta\leq1$, and $Q\geq1$, this gives
		\begin{equation}
			\alpha
			\leq
			\sqrt{\frac{\mu\delta^2}{320QL^3}}
			\leq
			\frac{1}{\sqrt{320}L}
			\leq
			\frac{1}{4L}.
			\label{tv_gt_stepsize_implies_smooth}
		\end{equation}
		Moreover, $4\leq\sqrt{320}\sqrt{Q}\,\kappa^{3/2}$ together with
		\eqref{tv_gt_alpha_bar_forms} yields
		\begin{equation}
			\bar\alpha
			=
			\frac{\delta}{\sqrt{320}\,L\sqrt{Q\kappa}}
			\leq
			\frac{\delta}{4\mu},
			\label{tv_gt_stepsize_implies_mu}
		\end{equation}
		and therefore
	\begin{equation}
		\alpha\leq\frac{1}{4L},
		\qquad
		\alpha\leq\frac{\delta}{4\mu}.
		\label{tv_gt_stepsize_consequences}
	\end{equation}
	For compactness, write
	\begin{equation}
		\label{tv_gt_c_definition}
		\chi
		\define
		\frac{10\alpha^2QL^2}{\delta}.
	\end{equation}
	By \eqref{tv_gt_stepsize_implies_smooth},
	Lemma~\ref{lemma:tv_gt_coupled} applies, using $V_k=
		\widetilde{X}_k+\omega E_k$,
		$\omega=
		\frac{16\alpha L}{\delta}$, and
	\eqref{tv_gt_coupled_recursion} gives
	\begin{align}
		V_{k+1}
		&\leq
		\bigl(1-\mu\alpha+\omega \chi\bigr)\widetilde{X}_k
		+\omega E_k
		+\left[
		\frac{3\alpha L}{2}
		-\frac{\omega\delta}{4}
		+\omega \chi
		\right]\widetilde E_k
		\nonumber\\
		&\quad
		+\frac{\alpha^2\sigma^2}{N}
		+\omega\alpha^2Q\sigma^2.
		\label{tv_gt_V_pre}
	\end{align}
	The stepsize condition also gives
	\begin{equation}
		\label{tv_gt_omega_c_small}
		\omega \chi
		=
		\frac{160\alpha^3QL^3}{\delta^2}
		\leq
		\frac{\mu\alpha}{2}.
	\end{equation}
	Hence, $
		1-\mu\alpha+\omega \chi
		\leq
		1-\frac{\mu\alpha}{2}$.
	Moreover, $\omega\delta/4=4\alpha L$, while
	\eqref{tv_gt_omega_c_small} and $\mu\leq L$ give
	$\omega \chi\leq\alpha L/2$. Therefore,
	\begin{equation}
		\label{tv_gt_s_coefficient}
		\frac{3\alpha L}{2}
		-\frac{\omega\delta}{4}
		+\omega \chi
		\leq
		-2\alpha L.
	\end{equation}
	Hence, using these in \eqref{tv_gt_V_pre}, we get
		\begin{align}
		V_{k+1}
		&\leq
		(1-\tfrac{\mu\alpha}{2})\widetilde{X}_k
		+\omega E_k
		-2\alpha L \widetilde E_k
		\nonumber\\
		&\quad
		+\frac{\alpha^2\sigma^2}{N}
		+\omega\alpha^2Q\sigma^2.
		\label{tv_gt_V_pre2}
	\end{align}
	Since $E_k\leq\widetilde E_k$ by
	\eqref{tv_gt_bs_sandwich},
	\begin{align}
		\omega E_k-2\alpha L\widetilde E_k
		&\leq
		\omega\left(1-\frac{2\alpha L}{\omega}\right)E_k
		\nonumber\\
		&=
		\omega\left(1-\frac{\delta}{8}\right)E_k
		\nonumber\\
		&\leq
		\omega\left(1-\frac{\mu\alpha}{2}\right)E_k,
		\label{tv_gt_b_coefficient}
	\end{align}
	where the last inequality follows from
	\eqref{tv_gt_stepsize_consequences}. Finally,
	\begin{equation}
		\label{tv_gt_weighted_noise}
		\omega\alpha^2Q\sigma^2
		=
		\frac{16\alpha^3LQ\sigma^2}{\delta}.
	\end{equation}
	Substituting 
	\eqref{tv_gt_b_coefficient} and
	\eqref{tv_gt_weighted_noise} into
	\eqref{tv_gt_V_pre2} proves
	\eqref{tv_gt_V_recursion}.
\end{proof}

\subsection{Proof of Corollary~\ref{corr:tv_gt_main}}
\label{app:corr:tv_gt_main}

For completeness, define the technical stepsize and initialization quantities
used in the proof:
\begin{subequations}
\begin{align}
	\bar\alpha
	&\define 
	\frac{\delta}{\sqrt{320}\,L\sqrt{Q\kappa}} 	\label{tv_gt_alpha_bar}
 \\
 C_0&\define  \widetilde{X}_0
 +
 \frac{16\bar\alpha L}{\delta N}
 \|\widehat{\x}^0\|^2
 +
 \frac{32\bar\alpha^3L}{\eta^2\delta N}
 \|\grad\f(\x^\star)\|^2 \\
		T_K
		&\define
		\max\left\{
		e,\,
		\frac{\mu^2NC_0K}{2\sigma^2}
		\right\}  	\label{tv_gt_T_K}
		\\
		\alpha_K
		&\define
		\min\left\{
		\frac{2\ln T_K}{\mu K},
		\bar\alpha
		\right\}. 	\label{tv_gt_alpha_K}
\end{align}	
\end{subequations}
Here,  the choice $T_K\geq e$ ensures $\ln T_K\geq1$ and hence a strictly
positive candidate stepsize.

\begin{proof}
	By Theorem~\ref{thm:tv_gt_coupled_contraction},
	\begin{equation}
		\label{tv_gt_V_recursion_corollary}
		V_{k+1}
		\leq
		\left(1-\frac{\mu\alpha}{2}\right)V_k
		+
		\frac{\alpha^2\sigma^2}{N}
		+
		\frac{16\alpha^3LQ\sigma^2}{\delta}
	\end{equation}
	holds whenever $0<\alpha\leq\bar\alpha$. Iterating \eqref{tv_gt_V_recursion_corollary} and using $
		\sum_{t=0}^{K-1}
		\left(1-\frac{\mu\alpha}{2}\right)^t
		\leq
		\frac{2}{\mu\alpha}$,
	together with
	$1-t\leq e^{-t}$, gives
	\begin{equation}
		\label{tv_gt_V_finite_corollary}
		V_K
		\leq
		\exp\left(-\frac{\mu\alpha K}{2}\right)V_0
		+
		\frac{2\alpha\sigma^2}{\mu N}
		+
		\frac{32\alpha^2LQ\sigma^2}{\mu\delta}.
	\end{equation}
	Since $\y^0=\zero$, \eqref{tv_gt_ytilde_definition} and
	\eqref{tv_gt_z_eta_choice} give
	$\mathbf z^0=(\alpha/\eta)\grad\f(\x^\star)$. Hence,
	using $\R_0\preceq\I$,
	\begin{equation}
		\label{tv_gt_b0_bound}
			 E_0
		\leq
		\frac1N\|\widehat{\x}^0\|^2
		+
		\frac{2\alpha^2}{\eta^2N}
		\|\grad\f(\x^\star)\|^2.
	\end{equation}
	Then $
	V_0
	\leq
	C_0$,
	and using $\widetilde X_K\leq V_K$ in
	\eqref{tv_gt_V_finite_corollary} gives
	\begin{equation}
		\label{tv_gt_generic_K_bound}
		\widetilde{X}_K
		\leq
		C_0
		\exp\left(-\frac{\mu\alpha K}{2}\right)
		+
		\frac{2\alpha\sigma^2}{\mu N}
		+
		\frac{32\alpha^2LQ\sigma^2}{\mu\delta}.
	\end{equation}
	Set $\alpha=\alpha_K$. Since
	$\alpha_K\leq2\ln T_K/(\mu K)$, the two stochastic terms in
	\eqref{tv_gt_generic_K_bound} satisfy
	\begin{subequations}
		\label{tv_gt_alphaK_stochastic_bounds}
		\begin{align}
			\frac{2\alpha_K\sigma^2}{\mu N}
			&\leq
			\frac{4\sigma^2\ln T_K}{\mu^2NK},
			\\
			\frac{32\alpha_K^2LQ\sigma^2}{\mu\delta}
			&\leq
			\frac{128LQ\sigma^2\ln^2T_K}
			{\mu^3\delta K^2}.
		\end{align}
	\end{subequations}
	If $\alpha_K=2\ln T_K/(\mu K)$, then
	\begin{equation}
		\label{tv_gt_initial_log_bound}
		C_0
		\exp\left(-\frac{\mu\alpha_KK}{2}\right)
		=
		\frac{C_0}{T_K}
		\overset{\eqref{tv_gt_T_K}}{\leq}
		\frac{2\sigma^2}{\mu^2NK}.
	\end{equation}
	If instead $\alpha_K=\bar\alpha$, we retain the exponentially decaying
	initialization term. Consequently,
	\begin{align}
		\widetilde{X}_K
		&\leq
		\frac{2(1+2\ln T_K)\sigma^2}{\mu^2NK}
		+
		\frac{128LQ\sigma^2\ln^2T_K}{\mu^3\delta K^2}
		\nonumber\\
		&\quad+
		C_0
		\exp\left(-\frac{\mu\bar\alpha K}{2}\right).
		\label{tv_gt_transient_rate}
	\end{align}
	It remains to expose the network dependence. From \eqref{tv_gt_delta_gap} and \eqref{tv_gt_Q_definition}
	\begin{equation}
		\label{tv_gt_delta_Q_orders}
		\delta
		=
		\Theta\left(\frac{1-\lambda}{\tau}\right),
		\qquad
		Q
		=
		\Theta\left(\frac{\tau^2}{(1-\lambda)^2}\right).
	\end{equation}
	More explicitly, since $\tau\geq1$ and $\Delta_\lambda\leq1$,
	\begin{equation}
		\label{tv_gt_Q_bounds_exact}
		\frac{32\tau^2}{\Delta_\lambda^2}
		\leq
		Q
		\leq
		\frac{33\tau^2}{\Delta_\lambda^2}.
	\end{equation}
	Combining \eqref{tv_gt_delta_gap} and
	\eqref{tv_gt_Q_bounds_exact} yields
	\begin{equation}
		\label{tv_gt_Q_delta_exact}
		16\frac{\tau^3}{\Delta_\lambda^3}
		\leq
		\frac{Q}{\delta}
		\leq
		33\frac{\tau^3}{\Delta_\lambda^3}.
	\end{equation}
	Thus the first two terms in \eqref{tv_gt_transient_rate} give the first
	two terms in \eqref{tv_gt_transient_rate_explicit}, up to logarithmic
	factors and constants depending on $\mu$ and $L$.	For the initialization term, using $\delta\geq\Delta_\lambda/\tau$ and
	$Q\leq33\tau^2/\Delta_\lambda^2$ from \eqref{tv_gt_Q_bounds_exact},
	\begin{equation}
		\label{tv_gt_alpha_inverse_exact}
		\frac{\mu\bar\alpha}{2}
		=
		\frac{\delta}{2\sqrt{320}\,\sqrt{Q}\,\kappa^{3/2}}
		\geq
		\frac{\Delta_\lambda^2}{2\sqrt{320\cdot33}\,\kappa^{3/2}\tau^2}
		\geq
		\frac{\Delta_\lambda^2}{206\,\kappa^{3/2}\tau^2},
	\end{equation}
	so
	$\exp(-\mu\bar\alpha K/2)
	\leq\exp\bigl(-c\,(1-\lambda)^2K/(\kappa^{3/2}\tau^2)\bigr)$
	with $c=1/206$. This proves the exponential term in
	\eqref{tv_gt_transient_rate_explicit}.

	Finally, we determine when the network-dependent and initialization terms
	become lower order than the centralized statistical term. Ignoring logarithmic
	factors, the first two stochastic terms in
	\eqref{tv_gt_transient_rate} scale as
	\[
	\frac{\sigma^2}{\mu^2 N K}
	\qquad\text{and}\qquad
	\frac{LQ\sigma^2}{\mu^3\delta K^2},
	\]
	respectively. Requiring the network-dependent term to be no larger than the
	centralized term gives
	\begin{align}
		\frac{LQ\sigma^2}{\mu^3\delta K^2}
		&\lesssim
		\frac{\sigma^2}{\mu^2 N K}
		\nonumber\\
		\Longleftrightarrow\qquad
		K
		&\gtrsim
		N\frac{L}{\mu}\frac{Q}{\delta}
		=
		N\kappa\frac{Q}{\delta}.
	\end{align}
	Since $
	\frac{Q}{\delta}
	=
	\Theta\left(
	\frac{\tau^3}{(1-\lambda)^3}
	\right)$,
	this yields $
		K
		=
		\widetilde{\mathcal O}\left(
		\frac{N\kappa\tau^3}{(1-\lambda)^3}
		\right)$. For the initialization term, we require
		\begin{equation}
			C_0
			\exp\left(-\frac{\mu\bar\alpha K}{2}\right)
			\lesssim
			\frac{\sigma^2}{\mu^2 N K}.
		\end{equation}
		Equivalently,
		\begin{equation}
			K
			\gtrsim
			\frac{1}{\mu\bar\alpha}
			\log\left(
			\frac{\mu^2 N C_0K}{\sigma^2}
			\right).
		\end{equation}
		Thus, up to logarithmic factors, the initialization transient is determined by
		$1/(\mu\bar\alpha)$. Using $
			\frac{1}{\mu\bar\alpha}
			=
			\Theta\left(
			\frac{\kappa^{3/2}\tau^2}{(1-\lambda)^2}
			\right)$,
		we obtain
		\begin{equation}
			\label{tv_gt_initial_transient_proof}
			K
			=
			\widetilde{\mathcal O}\left(
			\frac{\kappa^{3/2}\tau^2}{(1-\lambda)^2}
			\right).
		\end{equation}
	Combining the two requirements gives
	\eqref{tv_gt_total_transient}.
\end{proof}
\subsection{Proof of Corollary~\ref{corr:tv_gt_last_iterate_consensus}}
\label{app:tv_gt_last_iterate_consensus}
\begin{proof}
	Define
	\[
	r\define 1-\frac{\delta}{8},
	\qquad
	\rho\define 1-\frac{\mu\alpha}{2}.
	\]
	Under the stepsize condition,
	\[
	\frac{10\alpha^2QL^2}{\delta}
	\leq
	\frac{\delta}{8}.
	\]
	Hence, using $E_k\leq\widetilde E_k$ in
	\eqref{tv_gt_coupled_b},
	\begin{equation}
		E_{k+1}
		\leq
		rE_k
		+
		\frac{10\alpha^2QL^2}{\delta}\widetilde X_k
		+
		\alpha^2Q\sigma^2.
		\label{tv_gt_consensus_pointwise_recursion}
	\end{equation}
	Theorem~\ref{thm:tv_gt_coupled_contraction} also gives
	\begin{equation}
		\widetilde X_k
		\leq
		\rho^k C_0
		+
		\frac{2\alpha\sigma^2}{\mu N}
		+
		\frac{32\alpha^2LQ\sigma^2}{\mu\delta}.
	\end{equation}
	Iterating \eqref{tv_gt_consensus_pointwise_recursion} gives
	\begin{align}
		E_K
		\leq{}&
		r^K E_0
		+
		\frac{10\alpha^2QL^2}{\delta}C_0
		\sum_{k=0}^{K-1}r^{K-1-k}\rho^k
		\nonumber\\
		&+
		\frac{8\alpha^2Q\sigma^2}{\delta}
		+
		\frac{160\alpha^3QL^2\sigma^2}{\mu N\delta^2}
		+
		\frac{2560\alpha^4Q^2L^3\sigma^2}{\mu\delta^3}.
		\label{tv_gt_consensus_iterated}
	\end{align}
	By \eqref{tv_gt_stepsize_consequences},
	$\mu\alpha/2\leq\delta/8$, so $r\leq\rho$ and
	\[
	\sum_{k=0}^{K-1}r^{K-1-k}\rho^k
	\leq
	K\rho^{K-1}.
	\]
	Set $\alpha=\alpha_K$. Since
	$\alpha_K\leq2\ln T_K/(\mu K)$ and $\alpha_K\leq\bar\alpha$,
	the three noise terms in \eqref{tv_gt_consensus_iterated} satisfy
	\begin{align}
		&\frac{8\alpha_K^2Q\sigma^2}{\delta}
		+
		\frac{160\alpha_K^3QL^2\sigma^2}{\mu N\delta^2}
		+
		\frac{2560\alpha_K^4Q^2L^3\sigma^2}{\mu\delta^3}
		\nonumber\\
		&\qquad=
		\widetilde{\mathcal O}\left(
		\frac{Q\sigma^2}{\delta K^2}
		+
		\frac{\sqrt Q\,\sigma^2}{\delta N K^2}
		\right)
		=
		\widetilde{\mathcal O}\left(
		\frac{\tau^3\sigma^2}
		{(1-\lambda)^3K^2}
		\right),
		\label{tv_gt_consensus_noise_bound}
	\end{align}
	where we used
	\[
	\frac{160\bar\alpha QL^2}{\mu\delta^2}
	=
	\frac{\sqrt{80Q\kappa}}{\delta},
	\qquad
	\frac{2560\bar\alpha^2Q^2L^3}{\mu\delta^3}
	=
	\frac{8Q}{\delta},
	\]
	and \eqref{tv_gt_Q_delta_exact}.
	For the initialization term,
	$1-t\leq e^{-t}$ and
	$\sup_{x>0}x^2e^{-x}=4e^{-2}$ give
	\begin{equation}
		r^KE_0
		\leq
		\frac{256e^{-2}E_0}{\delta^2K^2}
		=
		\mathcal O\left(
		\frac{\tau^2E_0}{(1-\lambda)^2K^2}
		\right).
		\label{tv_gt_consensus_initial_bound}
	\end{equation}
	It remains to bound the cross term. Since
	\[
	\frac{10\alpha_K^2QL^2}{\delta}
	\leq
	\frac{\delta}{32\kappa}
	\leq1,
	\]
	if $\alpha_K=2\ln T_K/(\mu K)$, then
	\[
	\rho^{K-1}
	\leq
	\frac{e^{1/8}}{T_K},
	\]
	and hence, by \eqref{tv_gt_T_K},
	\begin{equation}
		\frac{10\alpha_K^2QL^2}{\delta}
		C_0K\rho^{K-1}
		=
		\widetilde{\mathcal O}\left(
		\frac{\tau^3\sigma^2}
		{(1-\lambda)^3K^2}
		\right).
	\end{equation}
	If $\alpha_K=\bar\alpha$, then
	\begin{equation}
		\frac{10\alpha_K^2QL^2}{\delta}
		C_0K\rho^{K-1}
		\leq
		e^{1/8}C_0K
		\exp\left(
		-\frac{(1-\lambda)^2K}
		{206\,\kappa^{3/2}\tau^2}
		\right).
	\end{equation}
	Combining these bounds with
	$\frac1N\Ex\|\widehat{\x}^K\|_{\R_K}^2\leq E_K$
	proves \eqref{tv_gt_last_iterate_consensus_bound}.
\end{proof}

\subsection{Proof of Lemma \ref{lemma:tv_gt_convex_coupled}}
\label{app:tv_gt_convex_coupled}

\begin{proof}
	Lemma~\ref{lemma:tv_gt_average} remains valid after setting
	$\mu=0$. Therefore,
	\begin{equation}
		\label{tv_gt_cvx_centroid_from_average}
		\widetilde X_{k+1}
		\leq
		\widetilde X_k
		-\alpha\widetilde F_k
		+\frac{3\alpha L}{2N}
		\Ex\|\widehat{\x}^k\|^2
		+\frac{\alpha^2\sigma^2}{N}.
	\end{equation}
	Using
	$\frac1N\Ex\|\widehat{\x}^k\|^2\leq\widetilde E_k$ from
	\eqref{tv_gt_bs_sandwich} proves
	\eqref{tv_gt_cvx_centroid_recursion}.
	
	Now for the disagreement recursion, we know that for every convex $L$-smooth function $h$ and all $x,y$,
	\begin{equation}
		\label{tv_gt_cvx_smooth_convex_gradient}
		\|\grad h(x)-\grad h(y)\|^2
		\leq
		2L\left(
		 h(x)- h(y)
		-\left\langle\grad h(y),x-y\right\rangle
		\right).
	\end{equation}
	Recall that
	$\u^k=\grad\f(\x^k)-\grad\f(\x^\star)$. Splitting each local
	gradient difference at $\bar x^k$ and using
	$\|p+q\|^2\leq2\|p\|^2+2\|q\|^2$ gives
	\begin{align}
		\|\u^k\|^2
		&\leq
		2\sum_{i=1}^N
		\|\grad f_i(x_i^k)-\grad f_i(\bar x^k)\|^2
		\nonumber\\
		&\quad+
		2\sum_{i=1}^N
		\|\grad f_i(\bar x^k)-\grad f_i(x^\star)\|^2.
		\label{tv_gt_cvx_u_split}
	\end{align}
	The first sum satisfies
	\begin{equation}
		\label{tv_gt_cvx_u_first}
		2\sum_{i=1}^N
		\|\grad f_i(x_i^k)-\grad f_i(\bar x^k)\|^2
		\leq
		2L^2\|\widehat{\x}^k\|^2.
	\end{equation}
	Applying \eqref{tv_gt_cvx_smooth_convex_gradient} to the second sum
	with $x=\bar x^k$ and $y=x^\star$ gives
	\begin{align}
		&2\sum_{i=1}^N
		\|\grad f_i(\bar x^k)-\grad f_i(x^\star)\|^2
		\nonumber\\
		&\qquad\leq
		4L\sum_{i=1}^N
		\left(
		f_i(\bar x^k)-f_i(x^\star)
		-\left\langle
		\grad f_i(x^\star),\bar x^k-x^\star
		\right\rangle
		\right)
		\nonumber\\
		&\qquad=
		4NL\bigl(f(\bar x^k)-f(x^\star)\bigr),
		\label{tv_gt_cvx_u_second}
	\end{align}
	where we used
	$\sum_{i=1}^N\grad f_i(x^\star)
	=N\grad f(x^\star)=0$.
	Combining \eqref{tv_gt_cvx_u_split}--\eqref{tv_gt_cvx_u_second},
	taking expectations, and using
	$\frac1N\Ex\|\widehat{\x}^k\|^2\leq\widetilde E_k$ yields
	\begin{equation}
		\label{tv_gt_cvx_u_bound}
		\frac1N\Ex\|\u^k\|^2
		\leq
		2L^2\widetilde E_k
		+
		4L\widetilde F_k.
	\end{equation}
	The argument leading to \eqref{tv_gt_network_pre_coupled} uses only
	the network and noise assumptions and therefore remains unchanged.
	Taking total expectation in \eqref{tv_gt_network_pre_coupled} and
	dividing by $N$ gives
	\begin{equation}
		\label{tv_gt_cvx_network_pre}
		E_{k+1}
		\leq
		E_k
		-\frac{\delta}{4}\widetilde E_k
		+\frac{5\alpha^2Q}{\delta N}\Ex\|\u^k\|^2
		+\alpha^2Q\sigma^2.
	\end{equation}
	Substituting \eqref{tv_gt_cvx_u_bound} into
	\eqref{tv_gt_cvx_network_pre} proves
	\eqref{tv_gt_cvx_disagreement_recursion}.
\end{proof}

\subsection{Proof of Theorem \ref{thm:tv_gt_convex}}
\label{app:tv_gt_convex_thm}
\begin{proof}
	Since $\delta\leq1$ and $Q\geq1$, the stepsize condition
	\eqref{tv_gt_cvx_alpha_bnd} implies $\alpha\leq1/(4L)$, so
	Lemma~\ref{lemma:tv_gt_convex_coupled} applies.	Summing \eqref{tv_gt_cvx_centroid_recursion} from $k=0$ to $K-1$,
	using $\widetilde X_K\geq0$, and dividing by $\alpha K$ gives
	\begin{equation}
		\label{tv_gt_cvx_H_pre}
		\frac1K\sum_{k=0}^{K-1}\widetilde F_k
		\leq
		\frac{\widetilde X_0}{\alpha K}
		+
		\frac{3L}{2K}\sum_{k=0}^{K-1}\widetilde E_k
		+
		\frac{\alpha\sigma^2}{N}.
	\end{equation}
	Moreover, \eqref{tv_gt_cvx_alpha_bnd} implies
	\begin{equation}
		\label{tv_gt_cvx_network_absorb_one}
		\frac{10\alpha^2QL^2}{\delta}
		\leq
		\frac{\delta}{48}
		\leq
		\frac{\delta}{8}.
	\end{equation}
	Therefore, \eqref{tv_gt_cvx_disagreement_recursion} gives
	\begin{equation}
		\label{tv_gt_cvx_b_simplified}
		E_{k+1}
		\leq
		E_k
		-\frac{\delta}{8}\widetilde E_k
		+\frac{20\alpha^2QL}{\delta}\widetilde F_k
		+\alpha^2Q\sigma^2.
	\end{equation}
	Summing, using $E_K\geq0$, and dividing by $K$ yields
	\begin{equation}
		\label{tv_gt_cvx_S_pre}
		\frac1K\sum_{k=0}^{K-1}\widetilde E_k
		\leq
		\frac{8E_0}{\delta K}
		+
		\frac{160\alpha^2QL}{\delta^2K}
		\sum_{k=0}^{K-1}\widetilde F_k
		+
		\frac{8\alpha^2Q\sigma^2}{\delta}.
	\end{equation}
	Substituting \eqref{tv_gt_cvx_S_pre} into
	\eqref{tv_gt_cvx_H_pre} gives
	\begin{align}
		\frac1K\sum_{k=0}^{K-1}\widetilde F_k
		\leq{}&
		\frac{\widetilde X_0}{\alpha K}
		+
		\frac{12LE_0}{\delta K}
		+
		\frac{240\alpha^2QL^2}{\delta^2K}
		\sum_{k=0}^{K-1}\widetilde F_k
		\nonumber\\
		&+
		\frac{12\alpha^2QL\sigma^2}{\delta}
		+
		\frac{\alpha\sigma^2}{N}.
		\label{tv_gt_cvx_H_absorb_pre}
	\end{align}
	Again by \eqref{tv_gt_cvx_alpha_bnd},
	\begin{equation}
		\label{tv_gt_cvx_H_absorb}
		\frac{240\alpha^2QL^2}{\delta^2}
		\leq
		\frac12.
	\end{equation}
	Hence
	\begin{equation}
		\label{tv_gt_cvx_H_bound}
		\frac1K\sum_{k=0}^{K-1}\widetilde F_k
		\leq
		\frac{2\widetilde X_0}{\alpha K}
		+
		\frac{24LE_0}{\delta K}
		+
		\frac{24\alpha^2QL\sigma^2}{\delta}
		+
		\frac{2\alpha\sigma^2}{N}.
	\end{equation}
	Multiplying \eqref{tv_gt_cvx_S_pre} by $L$ and using
	\begin{equation}
		\frac{160\alpha^2QL^2}{\delta^2}
		\leq
		\frac13
	\end{equation}
	gives
	\begin{equation}
		\label{tv_gt_cvx_LS_bound}
		\frac{L}{K}\sum_{k=0}^{K-1}\widetilde E_k
		\leq
		\frac{8LE_0}{\delta K}
		+
		\frac{1}{3K}\sum_{k=0}^{K-1}\widetilde F_k
		+
		\frac{8\alpha^2QL\sigma^2}{\delta}.
	\end{equation}
	Adding $\frac1K\sum_{k=0}^{K-1}\widetilde F_k$ to
	\eqref{tv_gt_cvx_LS_bound} and then applying
	\eqref{tv_gt_cvx_H_bound} yields
	\begin{align}
		&\frac1K\sum_{k=0}^{K-1}
		\left(
		\widetilde F_k
		+
		L\widetilde E_k
		\right)
		\nonumber\\
		&\qquad\leq
		\frac{8\widetilde X_0}{3\alpha K}
		+
		\frac{40LE_0}{\delta K}
		+
		\frac{8\alpha\sigma^2}{3N}
		+
		\frac{40\alpha^2QL\sigma^2}{\delta}.
		\label{tv_gt_cvx_H_plus_LS}
	\end{align}
	Finally,
	$\widetilde F_k=\Ex[f(\bar x^k)-f(x^\star)]$ and
	$\frac1N\Ex\|\widehat{\x}^k\|^2\leq\widetilde E_k$, which proves
	\eqref{tv_gt_cvx_finite_horizon}.
\end{proof}

\subsection{Proof of Corollary \ref{corr:tv_gt_convex_rate}}
\label{app:tv_gt_convex_rate}

\begin{proof}
	Since $\y^0=\zero$, definitions \eqref{tv_gt_ytilde_definition} and
	\eqref{tv_gt_z_eta_choice} give
	\begin{equation}
		\label{tv_gt_cvx_z0}
		\mathbf z^0
		=
		\frac{\alpha}{\eta}\grad\f(\x^\star).
	\end{equation}
	Hence
	\begin{equation}
		\label{tv_gt_cvx_E0_bound}
		E_0
		=
		\frac1N
		\left(
		\|\widehat{\x}^0\|_{\R_0}^2
		+
		\frac{2\alpha^2}{\eta^2}
		\|\grad\f(\x^\star)\|_{\R_0}^2
		\right)
		\leq
		\bar E_0,
	\end{equation}
	for every $\alpha\leq\alpha_{\rm cvx}$.	Choose
	\begin{equation}
		\label{tv_gt_cvx_alpha_K}
		\alpha_K
		\define
		\min\left\{
		\sqrt{\frac{N\widetilde X_0}{\sigma^2K}},
		\left(
		\frac{\widetilde X_0\delta}
		{15QL\sigma^2K}
		\right)^{1/3},
		\alpha_{\rm cvx}
		\right\}.
	\end{equation}
	Using Theorem~\ref{thm:tv_gt_convex} and
	\eqref{tv_gt_cvx_E0_bound}, we have
	\begin{equation}
		\label{tv_gt_cvx_tuning_pre}
		\mathcal R_K
		\leq
		\frac{A}{\alpha_K}
		+
		B\alpha_K
		+
		C\alpha_K^2
		+
		\frac{40L\bar E_0}{\delta K},
	\end{equation}
	where $\mathcal R_K$ denotes the left-hand side of
	\eqref{tv_gt_cvx_finite_horizon} and
	\[
	A\define\frac{8\widetilde X_0}{3K},
	\qquad
	B\define\frac{8\sigma^2}{3N},
	\qquad
	C\define\frac{40QL\sigma^2}{\delta}.
	\]
	The choice \eqref{tv_gt_cvx_alpha_K} is equivalently
	\[
	\alpha_K
	=
	\min\left\{
	\sqrt{\frac AB},
	\left(\frac AC\right)^{1/3},
	\alpha_{\rm cvx}
	\right\}.
	\]
	Therefore,
	\[
	B\alpha_K\leq\sqrt{AB},
	\qquad
	C\alpha_K^2\leq A^{2/3}C^{1/3},
	\]
	and
	\[
	\frac{A}{\alpha_K}
	\leq
	\sqrt{AB}
	+
	A^{2/3}C^{1/3}
	+
	\frac{A}{\alpha_{\rm cvx}}.
	\]
	Substituting these bounds into \eqref{tv_gt_cvx_tuning_pre} gives
	\begin{equation}
		\mathcal R_K
		\leq
		2\sqrt{AB}
		+
		2A^{2/3}C^{1/3}
		+
		\frac{A}{\alpha_{\rm cvx}}
		+
		\frac{40L\bar E_0}{\delta K}.
		\label{tv_gt_cvx_tuned_pre_network}
	\end{equation}	
	Finally,
	\[
	\frac{Q}{\delta}
	=
	\Theta\left(
	\frac{\tau^3}{(1-\lambda)^3}
	\right),
	\qquad
	\frac1{\alpha_{\rm cvx}}
	=
	\Theta\left(
	\frac{L\tau^2}{(1-\lambda)^2}
	\right),
	\qquad
	\frac1\delta
	=
	\Theta\left(
	\frac{\tau}{1-\lambda}
	\right).
	\]
	Applying these relations to
	\eqref{tv_gt_cvx_tuned_pre_network} proves
	\eqref{tv_gt_cvx_network_rate}.
\end{proof}

\small

\bibliographystyle{ieeetr}
\bibliography{ref_tv_gt}

\end{document}